\documentclass[11pt]{amsart}
\usepackage{amsmath}
\usepackage{amsthm}
\usepackage{amssymb}
\usepackage{a4wide}
\usepackage{clock}

\usepackage{hyperref}
\usepackage{amsfonts}

\newtheorem{lemma}{Lemma}[section]

\newtheorem{prop}{Proposition}[section]
\newtheorem{teo}{Theorem}[section]

\newtheorem{remark}{Remark}[section]

\def\zz{\mathbb{Z}}
\def\rr{\mathbb{R}}

\def\cn{\mathbb{C}}

\def\rr{\mathbb{R}}

\def\what{\widehat}

\def\eps{\varepsilon}
\def\tq{\tilde q}
\def\tr{\tilde r}

\def\th{{\bf T}_h}
\def\tph{{\bf T}_{4h}}
\def\eh{{\bf E}_h}
\def\sg{\sigma}

\def\tuh{\tilde{u}_h}

\def\bpi{\boldsymbol{\Pi}}
\def\err{{\rm err}_h}
\def\errlin{{\rm err}_h^{lin}}
\def\errnon{{\rm err}_h^{non}}
\def\hat{\what}
\def\tvarphi{\tilde \varphi}
\def\tu{\tilde u}

\title[Conv. rates for dispersive approximations]{Convergence rates for dispersive approximation schemes to nonlinear Schr\"odinger
equations}
\author[L. I. Ignat and E.  Zuazua]
{Liviu I. Ignat \and Enrique Zuazua}

\address{L. I. Ignat
\hfill\break\indent Institute of Mathematics ``Simion Stoilow'' of the Romanian Academy\\21 Calea Grivitei Street \\010702 Bucharest \\ Romania 
\hfill\break\indent \and
\hfill\break\indent 
 BCAM - Basque Center for Applied Mathematics,\\
Bizkaia Technology Park, Building 500 Derio, Basque Country, Spain.}
 \email{{\tt
liviu.ignat@gmail.com}\hfill\break\indent  {\it Web page: }{\tt
http://www.imar.ro/\~\,lignat}}

\address{E. Zuazua
\hfill\break\indent  Ikerbasque, Basque Foundation for Science,\\
 Alameda Urquijo 36-5, Plaza Bizkaia, 48011, Bilbao, Basque Country, Spain
\hfill\break\indent \and
\hfill\break\indent 
 BCAM - Basque Center for Applied Mathematics,\\
Bizkaia Technology Park, Building 500 Derio, Basque Country, Spain.}
  \email{{\tt zuazua@bcamath.org }
\hfill\break\indent {\it Web page: }{\tt http://www.bcamath.org/zuazua/}}

\usepackage{pdfsync}
\begin{document}

\begin{abstract}{This article is devoted to the analysis of the convergence rates of several numerical approximation
schemes for  linear and nonlinear Schr\"odinger equations on the real line. Recently, the authors have introduced
viscous and two-grid numerical approximation schemes that mimic at the discrete level the so-called Strichartz dispersive estimates
of the continuous Schr\"odinger equation. This allows to guarantee the convergence of numerical approximations
for initial data in $L^2(\rr)$, a fact that
can not be proved in the nonlinear setting for standard conservative schemes unless more regularity of the initial data is assumed.
In the present article we obtain explicit  convergence rates and prove that dispersive schemes fulfilling the Strichartz estimates
are better behaved  for $H^s(\rr)$ data if $0<s<1/2$. Indeed, while dispersive schemes ensure a polynomial convergence rate, non-dispersive ones only yield logarithmic decay rates.}

\end{abstract}

\maketitle


\section{Introduction}
Let us consider the linear (LSE) and the nonlinear (NSE) Schr\"{o}dinger equations:
\begin{equation}\label{sch1}
\left\{\begin{array}{l}
  iu_t+\partial_x^2 u = 0,  \,x\in \rr,\,t\neq 0,\\
  u(0,x)  =\varphi(x), \,x\in \rr
    \end{array}\right.
\end{equation}
and
\begin{equation}\label{scha869}
\left\{
\begin{array}{l}
  iu_t+\partial_x^2 u = f(u),  \,x\in \rr,\,t\neq 0,\\
  u(0,x)  =\varphi(x),\, x\in \rr,
  \end{array}\right.
\end{equation}
respectively.

The linear equation \eqref{sch1} is solved by $u(x,t)=S(t)\varphi$, where
$S(t)=e^{it\Delta}$  is the free Schr\"{o}dinger operator and has two
important properties. First, the conservation of the $L^2$-norm
\begin{equation}\label{energy}
\|u(t)\|_{L^2(\rr)}=\|\varphi \|_{L^2(\rr)}
\end{equation}
which shows that it is in fact a group of isometries in $L^2(\rr)$,
 and a dispersive estimate of the form:
\begin{equation}\label{linfty}
|S(t)\varphi(x)|=|u(t,x)|\leq \frac 1{(4\pi |t|)^{1/2}}\|\varphi \|_{L^1(\rr)}, \ x\in
\rr,\ t\neq 0.
\end{equation}


The space-time estimate
\begin{equation}\label{l6l6}
    \|S(\cdot)\varphi\|_{L^{6}(\rr,\,L^{6}(\rr))}\leq C\|\varphi\|_{L^2(\rr)},
\end{equation}
 due to  Strichartz \cite{0372.35001}, guarantees  that the
solutions decay  as $t$ becomes large and that they gain some
spatial integrability. 

Inequality (\ref{l6l6}) was generalized by Ginibre and Velo \cite{MR801582}. They proved:
\begin{equation}\label{dsch0001}
    \|S(\cdot)\varphi\|_{L^q(\rr,\,L^r(\rr))}\leq C(q)\|\varphi\|_{L^2(\rr)}
\end{equation}
 for  the so-called $1/2$-admissible pairs $(q,r)$. We recall that the exponent
pair $(q,r)$ is  $\alpha$-admissible (cf. \cite{0922.35028}) if  $2\leq q,r\leq \infty$,
$(q,r,\alpha)\neq (2,\infty,1)$ and
\begin{equation}\label{adm}
\frac 1 {q}=\alpha\left(\frac 12 -\frac 1r\right).
\end{equation}
We see that \eqref{l6l6} is a particular instance of \eqref{dsch0001} in which $\alpha=1/2$ and
$q=r=6$.


The extension of these estimates to the inhomogeneous linear Schr\"{o}dinger equation is due to Yajima
\cite{MR891945} and Cazenave and Weissler \cite{MR952091}. These estimates can also be
extended to a larger class of equations for which the Laplacian is replaced by any
self-adjoint operator such that the $L^\infty$-norm of the fundamental solution behaves
like $t^{-1/2}$, \cite{0922.35028}.

The Strichartz estimates  play an important role in the proof of the well-posedness of
the nonlinear Schr\"odinger equation. Typically they are used for nonlinearities for which the energy methods
fail to provide well-posedness results. In this way, Tsutsumi \cite{0638.35021} proved the existence and uniqueness
for $L^2(\rr)$-initial data for power-like nonlinearities $F(u)=|u|^{p}u$, in the range of exponents
$0\leq p\leq 4$. More precisely it was proved that the NSE is globally well posed in $L^\infty(\rr,\,L^2(\rr))\cap
L^q_{loc}(\rr,\,L^r(\rr))$, where $(q,r)$ is a $1/2$-admissible pair depending  on the exponent
$p$. This result was complemented by Cazenave and Weissler \cite{MR1021011} who
proved the local existence
in the critical case $p=4 $. The case of $H^1$-solutions was analyzed by Baillon,
Cazenave and Figueira \cite{MR0433025}, Lin and Strauss \cite{MR515228}, Ginibre and Velo
\cite{MR498408, MR533218}, Cazenave \cite{MR559676}, and, in a more general context, by
Kato \cite{MR877998, MR1037322}.

This analysis has been extended to semi-discrete numerical schemes
for Schr\"odinger equations by Ignat and Zuazua in \cite{1063.35016}, \cite{liv3}, \cite{liv-siam}.
In these articles  it was first pointed out that conservative numerical schemes
 often fail to be dispersive, in the sense that numerical solutions do not fulfill the integrability properties above. This is due to the pathological behavior of high frequency spurious numerical solutions.
 Then  several numerical schemes were developed fulfilling the dispersive properties,
uniformly in the mesh-parameter. In the sequel these schemes will be referred to as being \emph{dispersive}. As proved in those articles these
schemes may be used in the nonlinear context  to prove convergence towards the solutions of the NSE,                 for the range of exponents $p$ and the functional setting
 above.
 The analysis of  fully discrete
schemes  was later developed  in \cite{liv-m3as} where
 necessary and sufficient conditions were given guaranteeing that the dispersive
properties of the continuous model are maintained uniformly with respect to the mesh-size parameters
 at the discrete level. The present paper is devoted to further analyze the convergence of these numerical schemes, the main goal being the obtention of  convergence rates.

Despite of the fact that non-dispersive schemes  (in the sense that they do not satisfy the discrete analogue of \eqref{l6l6})  can not be applied directly in the $L^2$-setting for nonlinear equations
one could still use them
 by first approximating the $L^2$-initial data by smooth ones. 
This paper is devoted to prove that, even if this is done, 
 dispersive schemes are better behaved than the non-dispersive
 ones 
 in what concerns the order of convergence for rough initial data.
 
The main results of the paper are as follows.
In Theorem \ref{est.lin} we prove that the error
 committed when the LSE is approximated by a dispersive numerical scheme in the  $L^q(0,T;\,l^r(h\zz))$-norms is of the same order as the one classical consistency+stability analysis yields.
Using the ideas of \cite{MR0461121}, Ch. ~6 we can also estimate the error in the $L^q(0,T;\,l^r(h\zz))$-norms, $r>2$, for  non-dispersive schemes; for
example for the classical three-point second order approximation  of the laplace operator. In this case,
in contrast with the good properties of dispersive schemes, for  $H^s(\rr)$-initial data with small $s$, $1/2-1/r\leq s\leq 4+1/2-1/r$, the  error losses a factor of order $h^{3/2(1/2-1/r)}$ with respect to the case $L^\infty (0,T;\,l^2(h\zz))$ which can be handled by classical energy methods  (see Example 1 in Section \ref{examples}).
Summarizing, we see that dispersive property of numerical schemes is needed to guarantee that the convergence rate of numerical solution is kept in the spaces $L^q(0,T;\,l^r(h\zz))$.

%

In the the context of the  NSE we prove that the dispersive methods introduced in this paper converge to the solutions of NSE with the same order as in  the linear problem.  
 To be more precise,  in Theorem \ref{main} we prove
 a  polynomial order of convergence, $h^{s/2}$, in the case of a dispersive
approximation scheme of order two for the laplace operator for initial data $H^s(\rr^d)$ when $0<s<4$.
In the case of the classical non-dispersive schemes this convergence rate can only be guaranteed for smooth enough initial data, $H^s(\rr)$, $1/2<s<4$ (see Theorem \ref{regularizare}).

In Section \ref{nondisipative} we show that non-dispersive numerical schemes with rough data behaves badly. Indeed, 
 when using non-dispersive numerical schemes, combined with a  $H^1(\rr)$-approximation of the initial data $\varphi\in H^s(\rr)\backslash H^1(\rr)$, one gets an order of convergence $|\log h|^{-s/(1-s)}$
 which is much weeker than the $h^{s/2}$-one that dispersive schemes ensure.

The paper is organized as follows. In Section \ref{lin} we first obtain a quite general result
which allows us to estimate the difference of two families of operators that admit Strichartz estimates.
We then particularize it to operators acting on  discrete spaces $l^p(h\zz)$, obtaining results which
will be used in the following sections to get the order of convergence for approximations of the NSE.
In Section \ref{dis.lse}  and Section \ref{two.grid.section} we revisit the dispersive schemes for LSE introduced in \cite{liv2,1063.35016,liv3,liv-siam}
which are based, respectively, on the use of artificial numerical viscosity and a two-grid preconditioning 
technique of the initial data.

Section \ref{non} is devoted to analyze
approximations of the NSE based on the dispersive schemes analyzed in previous sections.  Section \ref{nondisipative} contains classical
material on conservative schemes that we include here in order to emphasize the advantages of the dispersive methods.  Finally, Section
\ref{putos.lemas} contains some technical results  used along the paper.

The analysis in this paper can be extended to fully discrete dispersive schemes introduced and analyzed in
\cite{liv-m3as} and to the multidimensional case. 
However, several technical aspects need to be dealt with carefully. In particular,  one has to take care of the well-posedness of the NSE (see \cite{1055.35003}). Furthermore,  suitable versions of the technical  harmonic analysis results employed in the paper (see, for instance, Section \ref{putos.lemas})  would also be needed (see \cite{grafakos}). This will be the object of future work.

Our methods use Fourier analysis techniques in an essential manner. Adapting this theory to numerical approximation schemes in non-regular meshes is by now a completely open subject.

%
%

\section{Estimates on  linear semigroups} \label{lin}
\setcounter{equation}{0}
In this section we will obtain $L^q_tL^r_x$ estimates for the difference of two semigroups $S_A(t)$ and $S_B(t)$  which admit Strichartz estimates. Once
this result is obtained in an abstract setting we particularize it to the discrete spaces $l^p(h\zz)$.

\subsection{An abstract result}
First we state  a well-known result by Keel and Tao \cite{0922.35028}.
\begin{prop}(\cite{0922.35028}, Theorem 1.2)\label{strfil}
Let $H$ be a Hilbert space, $(X,dx)$ be a measure space and
$U(t):H\rightarrow L^2(X)$ be a one parameter family of mappings with $t\in \rr$,
which obey the energy estimate
\begin{equation}\label{l2-stable}
    \|U(t)f\|_{L^2(X)}\leq C\|f\|_{H}
\end{equation}
and the decay estimate
\begin{equation}\label{linfty.dis}
    \|U(t)U(s)^*g\|_{L^\infty(X)}\leq C|t-s|^{-\alpha}\|g\|_{L^1(X)}
\end{equation}
for some $\alpha>0$. Then
\begin{equation}\label{dis.stric}
   \|U(t)f\|_{L^q(\rr,\,L^r(X))}\leq C\|f\|_{H},
\end{equation}
\begin{equation}\label{filsch401}
  \left\|\int_\rr (U(s))^*F(s,\cdot))ds\right\|_{H}\leq
C\|F\|_{L^{q'}(\rr,\,L^{r'}(X))},
\end{equation}
\begin{equation}\label{dis.stric.inhomg}
    \left\|\int _0^t
U(t-s)F(s)ds\right\|_{L^q(\rr,\,L^{r}(X))}\leq
C\|F\|_{L^{\tq'}(\rr,\,L^{\tr '}(X))}
\end{equation}
for all $(q,r)$ and $(\tq,\tr)$, $\alpha$-admissible pairs.
\end{prop}

%

The following theorem provides the key estimate in obtaining the order of convergence when the LSE is approximated by a dispersive scheme.

\begin{teo}\label{diff}
Let $(X,dx)$ be a measure space,  $A:D(A)\rightarrow L^2(X)$, $B:D(B)\rightarrow L^2(X)$
two linear m-dissipative operators with $D(A)\hookrightarrow D(B)$ continuously and satisfying $AB=BA$.
Assume that $(S_A(t))_{t\geq 0}$ and $(S_B(t))_{t\geq 0}$  the semigroups
generated by  $A$ and $B$ satisfy  assumptions  \eqref{l2-stable} and \eqref{linfty.dis} with $H=L^2(X)$.
 Then  for any   two $\alpha$-admissible pairs $(q,r)$, $(\tq,\tr)$ the following hold:\\
 i)
There exists a positive constant $ C(q)$ such that
\begin{align}\label{dis.stric.2}
   \|S_A(t)\varphi-S_B(t)\varphi\|_{L^q(I,\,L^r(X))}\leq 
    C(q)\min\Big\{\|\varphi\|_{L^2(X)},|I|\|(A-B)\varphi\|_{L^2(X)}\Big\}
\end{align}
 for all bounded  intervals $I$ and $\varphi\in D(A)\cap D(B)$. \\
 ii) There exists a positive constant
$C(q,\tq)$ such that
\begin{align}\label{dis.stric.inhomg.2}
    \Big\|\int
    _{0}^t&S_A(t-s)f(s)ds-\int
    _{0}^tS_B(t-s)f(s)ds\Big\|_{L^q(I,\,L^r(X))}
    \\[10pt]
   \nonumber &  \leq C(q,\tq)\min\Big\{\|f\|_{L^{\tilde{q}'}(I,\,L^{\tilde{r}'}(X))},
   |I|\|(A-B)f\|_{L^{\tilde{q}'}(I,\,L^{\tilde{r}'}(X)) }\Big\}
\end{align}
  for all bounded  intervals $I$ and  $f\in L^{\tilde{q}'}(I,\,L^{\tilde{r}'}(X)) $ such that $(A-B)f\in L^{\tilde{q}'}(I,\,L^{\tilde{r}'}(X))$.

\end{teo}

\begin{proof}[Proof of Theorem \ref{diff}] Using that the operators $S_A$ and $S_B$ verify
hypotheses \eqref{l2-stable} and \eqref{linfty.dis} of  Proposition \ref{strfil} with $H=L^2(X)$, by \eqref{dis.stric}  we  obtain
\begin{equation}\label{diff.1}
\|S_A(t)\varphi-S_B(t)\varphi\|_{L^q(I,\,L^r(X))}\leq
    C(q)\|\varphi\|_{L^2(X)}
\end{equation}
and, by \eqref{dis.stric.inhomg},
\begin{equation}\label{diff.2}
 \Big\|\int
    _{0}^tS_A(t-s)f(s)ds-\int
    _{0}^tS_B(t-s)f(s)ds\Big\|_{L^q(\rr,\,L^r(X))}
   \leq C(q,\tq)\|f\|_{L^{\tilde{q}'}(\rr,\,L^{\tilde{r}'}(X))}.
\end{equation}

In view of \eqref{diff.1} and \eqref{diff.2} it is then sufficient to prove the following estimates:
\begin{equation}\label{est.diff.1}
   \|S_A(t)\varphi-S_B(t)\varphi\|_{L^q(I,\,L^r(X))}\leq C(q)|I|\|(A-B)\varphi\|_{L^2(X)}
\end{equation}
and
\begin{equation}\label{dis.stric.inhomg.2.1}
     \Big\|\int
    _{0}^tS_A(t-s)f(s)ds-\int
    _{0}^tS_B(t-s)f(s)ds\Big\|_{L^q(I,\,L^r(X))}
     \leq C(q,\tq)|I|
   \|(A-B)f\|_{L^{\tilde{q}'}(I,\,L^{\tilde{r}'}(X)) }.
\end{equation}

In the case of \eqref{est.diff.1}
 we write the difference $S_A(\cdot)-S_B(\cdot)$ as follows  
 \begin{equation}\label{semigrup-identity}
S_A(t)\varphi-S_B(t)\varphi=\int _0^t S_B(t-s)(A-B)S_A(s)\varphi  ds.
\end{equation}
 In order to justify this identity let us recall that for any $\varphi\in D(A)\hookrightarrow D(B)$ we have that
  $u(t)=S_A(t)\varphi \in C([0,\infty),D(A))\cap C^1([0,\infty),L^2(X))$ and $v(t)=S_B(t)\varphi \in C([0,\infty),D(B))\cap C^1([0,\infty),L^2(X))$ verify
 the systems
 $u_t=Au, u(0)=\varphi,$ and $v_t=Bv,v(0)=\varphi$ respectively. Thus $w=u-v\in C([0,\infty),D(B))\cap C^1([0,\infty),L^2(X))$ satisfy the system
 $w_t=Bw+(A-B)u, w(0)=0$. Since $(A-B)u\in C([0,\infty),L^2(X))$ we obtain that $w$ satisfies \eqref{semigrup-identity}.

Going back to \eqref{semigrup-identity} and using  that $A$ and $B$ commute we get the following identity which is the key of our estimates:
\begin{equation}\label{magic}
S_A(t)\varphi-S_B(t)\varphi=\int _0^t S_B(t-s)S_A(s)(A-B)\varphi  ds.
\end{equation}
We apply Proposition \ref{strfil} to the semigroup $S_B(\cdot)$ and  function $F(s)=S_A(s)(A-B)\varphi$ in this identity and,
by \eqref{dis.stric.inhomg} with $\tilde r=2$ and $\tilde q=\infty$, we get
\begin{align}\label{diff.3}
\|S_A(t)\varphi-S_B(t)\varphi\|_{L^q(I,\,L^r(X))}&\leq C(q) \|S_A(s)(A-B)\varphi\|_{L^1(I,\,L^2(X))}\\
\nonumber&\leq C(q)|I|\|(A-B)\varphi\|_{L^2(X)}.
\end{align}
Thus, \eqref{est.diff.1} is proved. As a consequence
\eqref{diff.1} and \eqref{est.diff.1} give us \eqref{dis.stric.2}.

\medskip
We now prove the inhomogenous estimate \eqref{dis.stric.inhomg.2.1}.
 Using again \eqref{magic} we have
$$S_A(t-s)f(s)-S_B(t-s)f(s)=\int_0^{t-s}S_B(t-s-\sg)S_A(\sg )(A-B)f(s)d\sg.$$
We integrate this identity in the $s$ variable. Applying Fubini's theorem on the triangle
$\{(s,\sg):\,0\leq s\leq t, 0\leq \sigma\leq t-s\}$ and using that $A$ and $B$ commute,
 we get:
\begin{align}\label{identity}
\Lambda f(t):=&\int _{0}^tS_A(t-s)f(s)ds-\int  _{0}^tS_B(t-s)f(s)ds\\
\nonumber&=\int_0^t \int_0^{t-s}S_B(t-s-\sg)S_A(\sg )(A-B)f(s)d\sg ds\\
\nonumber&=\int_0^t \int_0^{t-\sg}S_B(t-s-\sg)S_A(\sg )ds (A-B)f(s)d\sg\\
\nonumber&=\int_0^tS_A(\sg )\int _0^{t-\sg}S_B(t-s-\sg)(A-B)f(s)dsd\sigma\\
\nonumber&\mathop{=}^{\sg\rightarrow t-\sg} \int_0^t S_A(t-\sg) \int _0^{\sg} S_B(\sg-s)(A-B)f(s)dsd\sigma\\
\nonumber&=\int_0^t S_A(t-\sg)\Lambda_1 (A-B)f(\sg)d\sg
\end{align}
where
$$\Lambda_1 g(t)=\int_0^t S_B(t-\tau) g(\tau)d\tau.$$
Applying the inhomogeneous estimate \eqref{dis.stric.inhomg} to the operator $S_A(\cdot)$ with $(\tq',\tr')=(1,2)$ we obtain
\begin{equation}\label{pula.1}
\|\Lambda f\|_{L^q(I,\,L^r(X))}\leq C(q)\|\Lambda_1 (A-B)f\|_{L^1(I,\,L^2(X))}\leq C(q)|I| \|\Lambda_1 (A-B)f\|_{L^\infty(I,\,L^2(X))}.
\end{equation}
Using again \eqref{dis.stric.inhomg}  for the semigroup $S_B(\cdot)$, $F=(A-B)f$ and $(q,r)=(\infty,2)$ we get
\begin{equation}\label{pula.2}
\|\Lambda_1 (A-B)f\|_{L^\infty(I,\,L^2(X))}\leq C(\tq)\|(A-B)f\|_{L^{\tq'}(I,\, L^{\tr'}(X))}.
\end{equation}
Combining \eqref{pula.1} and \eqref{pula.2} we deduce  \eqref{dis.stric.inhomg.2.1}. Estimates \eqref{diff.2} and \eqref{dis.stric.inhomg.2.1}
 finish the proof.
\end{proof}

\begin{remark}
 We point out that, in the the proof of the following estimate
 $$\|S_A(t)\varphi-S_B(t)\varphi\|_{L^q(I,\,l^r(X))}\leq
    C(q)|I|\|(A-B)\varphi\|_{L^2(X)},$$
  in view of \eqref{magic} and \eqref{diff.3}, we do not need that the two operators $S_A(t)$ and $S_B(t)$ admit Strichartz  estimates. Indeed, it
  is sufficient  to assume that  only one of the involved operators admits Strichartz  estimates and the other one to be
   stable in $L^2(X)$.
\end{remark}

\subsection{Spaces and Notations.} In this section we  introduce the spaces we will use along the paper.
The computational mesh is $h\zz=\{jh:\,j\in\zz\}$ for some $h>0$ and the $l^p(h\zz)$ spaces are defined as follows:
$$l^p(h\zz)=\{\varphi:h\zz\rightarrow \cn:\, \|\varphi\|_{l^p(h\zz)}<\infty\}$$
where
$$\|\varphi\|_{l^p(h\zz)}=\left \{
\begin{array}{cc}
  \displaystyle\Big(h\sum_{j\in \zz}|u(jh)|^p\Big)^{1/p} & 1\leq p<\infty, \\[10pt]
   \displaystyle\sup_{j\in \zz} |u(jh)| & p=\infty.
\end{array}
\right.$$
On the Hilbert space $l^2(h\zz)$ we will consider the following scalar product
$$(u,v)_h={\rm{Re}}\Big(h\sum_{j\in \zz}u(jh)\overline{v(jh)}\Big).$$
When necessary, to simplify the presentation, we will write $(\varphi_j)_{j\in \zz}$ instead of
$(\varphi(jh))_{j\in \zz}$.

For a discrete function $\{\varphi(jh)\}_{j\in \zz}$ we denote by $\hat\varphi$ its discrete Fourier transform:
\begin{equation}\label{dft}
\hat\varphi(\xi)=h\sum_{j\in \zz} e^{-ij\xi h}\varphi(jh).
\end{equation}

For $s\geq 0$ and $1< p<\infty$, $W^{s,p}(\rr)$ denotes  the Sobolev space
$$W^{s,p}(\rr)=\{\varphi\in \mathcal{S}'(\rr):\,(I-\Delta)^{s/2}\varphi\in L^p(\rr) \}$$
with the norm
$$\|\varphi\|_{W^{s,p}(\rr)}=\|((1+|\xi|^2)^{s/2}\hat \varphi)^\vee\|_{L^p(\rr)}$$
and by $H^s(\rr)$ the Hilbert space $W^{s,2}(\rr)$.

The homogenous spaces $\dot W^{s,p}(\rr)$, $s\geq 0$ and $1\leq p<\infty$, are given by
$$\dot W^{s,p}(\rr)=\{\varphi\in \mathcal{S}'(\rr):\,(-\Delta)^{s/2}\varphi\in L^p(\rr) \}$$
endowed with the semi-norm
$$\|\varphi\|_{\dot W^{s,p}(\rr)}=\|(|\xi|^s\hat \varphi)^\vee\|_{L^p(\rr)}.$$ If $p=2$ we denote $\dot H^s(\rr)=\dot W^{s,2}(\rr).$

We will also use the Besov spaces both in the continuous and the discrete framework. It is convenient to
consider a function $\eta_0\in C_c(\rr)$ such that
$$\eta_0(\xi)=\left\{\begin{array}{ccc}
                     1 & \mbox{if}& |\xi|\leq 1, \\[3pt]
                     0 & \mbox{if}& |\xi|\geq 2,
                   \end{array}
\right.$$
 and to define the sequence $(\eta_j)_{j\geq 1}\in \mathcal{S(\rr)}$ by
 $$\eta_j=\eta_0\Big(\frac \xi {2^j}\Big)-\eta_0\Big(\frac \xi {2^{j-1}}\Big)$$
in order to define the Littlewood-Paley decomposition.
For any $j\geq 0$ we set the cut-off projectors, $P_j\varphi$, as follows
\begin{equation}\label{cutoff}
P_j \varphi=(\eta_j \hat\varphi)^\vee.
\end{equation}
We point out that these projectors can be defined both for functions of continuous and discrete variables by means of the classical and the semi-discrete Fourier transform.

Classical results  on Fourier multipliers, namely Marcinkiewicz's multiplier theorem, (see Theorem \ref{graf}) show the following uniform estimate on the projectors $P_j$: For all $p\in (1,\infty)$ there exists $c(p)$ such that
\begin{equation}\label{est.pj}
\|P_j \varphi\|_{L^p(\rr)}\leq c(p)\| \varphi\|_{L^p(\rr)},\, \forall\, \varphi\in L^p(\rr).
\end{equation}

We introduce the Besov spaces
 $B^{s}_{p,2}(\rr)$ for $1\leq p\leq \infty$ by $B^s_{p,2}=\{u\in \mathcal{S}'(\rr):\|u\|_{B^{s}_{p,2}(\rr)}<\infty\}$
with
$$\|u\|_{B^{s}_{p,2}(\rr)}=\|P_0 u\|_{L^p(\rr)}+\Big(\sum_{j=1}^\infty
2^{2sj}\|P_j u\|_{L^p(\rr)}^2\Big)^{1/2}.$$
Their discrete counterpart $B_{p,2}^s(h\zz)$ with $1< p< \infty$ and $s\in \rr$ is given by
$$B_{p,2}^s(h\zz)=\{u:\|u\|_{B_{p,2}^s(h\zz)}<\infty\},$$
with
\begin{equation}\label{norma.besov}
    \|u\|_{B_{p,2}^s(h\zz)}=\|P_0 u\|_{l^p(h\zz)}+\Big(\sum_{j=1}^\infty 2^{2js}\|P_j u\|_{l^p(h\zz)}^2\Big)^{1/2},
\end{equation}
where $P_j u$ given as in \eqref{cutoff}  are now defined  by means of the discrete Fourier transform of the discrete function $u:h \zz\rightarrow
\cn$.

We will also adapt well known results from harmonic analysis to the discrete framework. We recall now a result which
  goes back to Plancherel and Polya \cite{0018.15204} (see also \cite{MR1836633},  Theorem~17. p.~96, and the comments on  p.~182).
\begin{lemma}(\cite{0018.15204}, p.\,157)\label{app18}
For any $p\in (1,\infty)$ there exist two positive constants $A(p)$ and $B(p)$ such that the following holds for all
functions $f$ whose Fourier transform is supported on $[-\pi,\pi]$:
\begin{equation}\label{app16}
A(p) \sum _{m\in \zz}|f(m)|^p \leq \int _\rr  |f(x)|^p dx\leq B(p) \sum _{m\in \zz}|f({m})|^p.
\end{equation}
\end{lemma}

This result permits to show, by scaling,   that, for all $h>0$,
\begin{equation}\label{norm.equiv}
A(p)^{1/p}\|f\|_{l^p(h\zz)}\leq \|f\|_{L^p(\rr)}\leq B(p)^{1/p}\|f\|_{l^p(h\zz)}
\end{equation}
holds for all functions $f$ with their
 Fourier transform supported in $[-\pi/h,\pi/h]$. 
%


%
%
%

For the sake of completeness we state now the discrete version of the well known uniform $L^p$-estimate \eqref{est.pj} for
the cut-off projectors $P_j$.

\begin{lemma}
For any $p\in (1,\infty)$ there exists a positive constant $c(p)$ such that
\begin{equation}\label{est.projection}
  \|P_j\varphi\|_{l^p(h\zz)}\leq c(p)\|\varphi\|_{l^p(h\zz)}
\end{equation}
holds for all $\varphi\in l^p(h\zz)$, $j\geq 0$, uniformly in $h>0.$
\end{lemma}

\begin{proof}
For a given discrete function $\varphi$ we consider its interpolator $\tilde \varphi$ 
defined as follows:
$$\tilde \varphi(x)=\int _{-\pi/h}^{\pi/h} e^{ix\xi } \hat \varphi(\xi)d\xi.$$
Thus, by \eqref{norm.equiv} we obtain
$$  \|P_j\varphi\|_{l^p(h\zz)}\leq c(p)\|(P_j \varphi)\,\tilde{}\, \|_{L^p(\rr)}=
c(p)\|P_j \tilde\varphi \|_{L^p(\rr)}\leq c(p)\| \tilde \varphi\|_{L^p(\rr)}
\leq c(p)\|\varphi\|_{l^p(h\zz)}.$$
\end{proof}

We recall  the following lemma which is a consequence of the Paley-Littlewood decomposition in the $x$ variable and
Minkowski's inequality in the time variable.

\begin{lemma}(\cite{MR1715192}, Ch.~5, p.~113, Lemma 5.2) Let $\eta\in C_c^\infty(\rr)$ and $P_j$ be defined as in
\eqref{cutoff}. Then
\begin{equation}\label{paley.1}
    \|\psi\|_{L^q(\rr,\,L^r(\rr))}^2\lesssim \sum_{j\geq 0}\|P_j \psi\|_{L^q(\rr,\,L^r(\rr))}^2 \ \text{if}\
2\leq r< \infty \ \text{and}\ 2\leq q\leq \infty
\end{equation}
and
\begin{equation}\label{paley.2}
    \sum_{j\geq 0}\|P_j \psi\|_{L^q(\rr,\,L^r(\rr))}^2\lesssim  \|\psi\|_{L^q(\rr,\,L^r(\rr))}^2 \ \text{if}\
1\leq r< 2 \ \text{and}\ 1\leq q\leq 2
\end{equation}
hold for all $\psi\in L^q(\rr,\,L^r(\rr))$.
\end{lemma}
Applying the above result and Lemma \ref{app18} to functions with their Fourier transform supported in $[-\pi/h,\pi/h]$, as above,
we can obtain a similar result in a discrete framework.
\begin{lemma}\label{discrete.paley} Let $\eta\in C_c^\infty(\rr)$ and $P_j$ defined as in
\eqref{cutoff}. Then
\begin{equation}\label{paley-direct}
    \|\psi\|_{L^q(\rr,\,l^r(h\zz))}^2\lesssim \sum_{j\geq 0}\|P_j \psi\|_{L^q(\rr,\,l^r(h\zz))}^2 \ \text{if}\
2\leq r< \infty \ \text{and}\ 2\leq q\leq \infty
\end{equation}
and
\begin{equation}\label{inverse.paley}
    \sum_{j\geq 0}\|P_j \psi\|_{L^q(\rr,\,l^r(h\zz))}^2\lesssim  \|\psi\|_{L^q(\rr,\,l^r(h\zz))}^2 \ \text{if}\
1\leq r< 2 \ \text{and}\ 1\leq q\leq 2
\end{equation}
hold  for all $\psi\in L^q(\rr,\,l^r(h\zz))$, uniformly in $h>0$.
\end{lemma}

%
%
%
%
%
%
%
%
%
%

\subsection{Operators on $l^p(h\zz)$-spaces.}
In the following we  apply the results of the previous section to the particular case $X=h\zz$.
We consider operators $A_h$with symbol $a_h:[-\pi/h,\pi/h]\rightarrow \cn$ such that
$$(A_h\varphi)_j=\int _{-\pi/h}^{\pi/h} e^{ij\xi h}a_h(\xi)\hat \varphi (\xi)d\xi,\,j\in\zz.$$
Also we will consider the operator
$|\nabla|^s$ acting on discrete spaces $l^2(h\zz)$ whose symbol is given by $|\xi|^s$. 

The numerical schemes we shall consider, associated to regular meshes, will enter in this frame by means of the Fourier representation formula of solutions.


\begin{teo}\label{error}
Let $A_h,B_h:l^2(h\zz)\rightarrow l^2(h\zz)$ be two operators whose symbols are $a_h$ and $b_h$, $i b_h$ being a real function, such that
the semigroups they generate, $(S_{A_h}(t))_{t\geq 0}$ and $(S_{B_h}(t))_{t\geq 0}$, satisfy  assumptions  \eqref{l2-stable} and \eqref{linfty.dis} with some
 constant $C$, independent of $h$. Finally, assume that for some functions $\{\mu(k,h)\}_{k\in F}$, with $F$ a finite set, the following
 holds for all $\xi\in [-\pi/h,\pi/h]$:
 \begin{equation}\label{hyp.0}
    |a_h(\xi)-b_h(\xi)|\leq \sum _{k\in F}\mu(k,h)|\xi|^k.
\end{equation}
For any $s>0$, denoting
\begin{equation}\label{eps}
    \eps(s,h)=\sum _{k\in F} \mu(k,h)^{\min\{s/k,1\}},
\end{equation}
the following hold for all $(q,r)$, $(\tq,\tr)$, $\alpha$-admissible pairs:\\
a) There exists a positive constant $C(q)$ such that
\begin{equation}\label{est.200}
 \|S_ {A_h}(t)\varphi-S_{B_h}(t)\varphi\|_{L^q(I,\,l^r(h\zz))}\leq C(q)\eps(s,h)\max\{1,|I|\}\|\varphi\|_{B^s_{2,2}(h\zz)}
\end{equation}
holds for all $\varphi\in B^s_{2,2}(h\zz)$ uniformly in $h>0$.\\
b) There exists a positive constant $C(s,q,\tq)$ such that
\begin{align}\label{est.201}
     \Big\|\int
    _{0}^t S_{A_h}(t-\sg)f(\sg)d\sg-&\int    _{0}^tS_{B_h}(t-\sg)f(\sg)d\sg\Big\|_{L^q(I,\,l^r(h\zz))}\\[5pt]
  \nonumber  &\leq C(s,q,\tq)\eps(s,h)\max\{1,|I|\}\|f\|_{L^{\tilde{q}'}(I,\,B^s_{\tilde{r}',2}(h\zz))}
\end{align}
holds for all $f\in L^{\tilde{q}'}(I,\,B^s_{\tilde{r}',2}(h\zz))$.
\end{teo}

\begin{remark}
The assumption that the semigroups $(S_{A_h}(t))_{t\geq 0}$ and $(S_{B_h}(t))_{t\geq 0}$, satisfy    \eqref{l2-stable} and \eqref{linfty.dis} with some constant $C$, independent of $h$, means that both of them are $l^2(h\zz)$-stable with constants that are independent of $h$ and
that the corresponding numerical schemes are dispersive.

Taking into account that both operators, $A_h$ and $B_h$, commute in view that they are associated to their symbols, the 
 hypotheses of Theorem \ref{diff} are fulfilled. They also commute with $|\nabla|$ and $P_j$ which are also defined by a Fourier symbol.

Assumption \eqref{hyp.0} on the operators $A_h$ and $B_h$ implies
$$\|(A_h-B_h)\varphi\|_{l^2(h\zz)}\lesssim \sum_{k\in F} a(k,h)\||\nabla |^k \varphi\|_{l^2(h\zz)}.$$
However, this assumption is not sufficient to obtain a similar estimate in $l^r(h\zz)$-norms,
$r\neq 2$. As we will see this will be an inconvenient in obtaining \eqref{est.201} as a consequence of
\eqref{dis.stric.inhomg.2}.

The requirement that $ib_h$ is a real function is needed to assure that the semigroup generated by $B_h$, $S_{B_h}$, satisfies
$$S_{B_h}(t-\sg)=S_{B_h}(t)S_{B_h}(-\sg)=S_{B_h}(t)S_{B_h}(\sg)^*,$$
identity which will be used in the proof.

In Section \ref{dis.lse} we will give examples of operators $A_h$ and $B_h$ verifying these hypotheses.
In all our estimates we will choose $b_h(\xi)=i\xi^2$ , which is the symbol of the continuous Schr\"odinger semigroup.
\end{remark}



\begin{proof}[Proof of Theorem \ref{error}]We divide the proof in two steps corresponding to the proof of \eqref{est.200} and
\eqref{est.201} respectively.

{\bf Step I. Proof of \eqref{est.200}.} We apply inequality \eqref{paley.1} to the difference
$S_{A_h}(t)\varphi-S_{B_h}(t)\varphi$:
\begin{equation*}
    \|S_{A_h}(t)\varphi- S_{B_h}(t)\varphi\|_{L^q(I; \,l^r(h\zz))}\leq  \Big(\sum _{j\geq 0}
\|P_j S_{A_h}(t)\varphi-P_j S_{B_h}(t)\varphi\|_{L^q(I, \,l^r(h\zz))}^2 \Big)^{1/2}.
\end{equation*}
Using that $P_j$ commutes with $S_{A_h}(\cdot)$ and $S_{B_h}(\cdot)$ we get:
\begin{equation}\label{est.101}
    \|S_{A_h}(t)\varphi- S_{B_h}(t)\varphi\|_{L^q(I; \,l^r(h\zz))}\leq  \Big(\sum _{j\geq 0}
\| \big(S_{A_h}(t)- S_{B_h}(t)\big)P_j\varphi\|_{L^q(I, \,l^r(h\zz))}^2 \Big)^{1/2}.
\end{equation}

In order to evaluate each term in the right hand side of \eqref{est.101} we apply estimate \eqref{dis.stric.2} to the difference   $S_{A_h}(\cdot)-S_{B_h}(\cdot)$ when acting on each
projection $P_j\varphi$. 
Thus,  using hypothesis \eqref{hyp.0} we obtain:
\begin{align}
\label{est.normala} \|S_ {A_h}(t)P_j\varphi -S_{B_h}&(t)P_j\varphi\|_{L^q(I,\,l^r(h\zz))}\\
  \nonumber  &\leq C(q)\max\{|I|,1\}\min\{\|P_j\varphi\|_{l^2(h\zz)},\|(A_h-B_h)P_j\varphi\|_{l^2(h\zz)}\}\\
  \nonumber  &\leq  C(q)\max\{|I|,1\}\min\Big\{\|P_j\varphi\|_{l^2(h\zz)},\sum _{k\in F}\mu(k,h)\||\nabla|^kP_j\varphi\|_{l^2(h\zz)}\Big\}\\
  \nonumber  &\leq C(q)\max\{|I|,1\}\sum _{k\in F} \min\Big\{\|P_j\varphi\|_{l^2(h\zz)},\mu(k,h)2^{jk}\|P_j\varphi\|_{l^2(h\zz)}\Big\}\\
  \nonumber &\leq C(q)\max\{|I|,1\}\|P_j\varphi\|_{L^2(\rr)} \sum _{k\in F}\min\Big\{1,
\mu(k,h)2^{jk}\Big\}.
\end{align}
Going back to estimate \eqref{est.101} we get
\begin{align*}
\|S_{A_h}(t)\varphi-& S_{B_h}(t)\varphi\|_{L^q(I; \,l^r(h\zz))}\\
&\leq C(q)\max\{|I|,1\} \Big(\sum _{j\geq 0}\|P_j\varphi\|_{L^2(\rr)}^2 \sum _{k\in F}\min\Big\{1,
\mu^2(k,h)2^{2jk}\Big\}\Big)^{1/2}.
\end{align*}
We claim that for any $j\geq 0$ the following holds
\begin{equation}\label{est.100}
\sum _{k\in F}
\min\Big\{1,
\mu^2(k,h)2^{2jk}\Big\}\leq \sum _{k\in F} \mu(k,h)^{\min\{2s/k,2\}}2^{2js}
\end{equation}
for all $s>0$.

Assuming for the moment that the claim \eqref{est.100} is correct we deduce that
\begin{align*}\|S_{A_h}(t)\varphi- &S_{B_h}(t)\varphi\|_{L^q(I; \,l^r(h\zz))}\\
&\leq  C(q)\max\{|I|,1\} \Big(\sum _{k\in F}\sum _{j\geq 0} \mu(k,h)^{\min\{2s/k,2\}}2^{2js}
\|P_j\varphi\|_{l^2(h\zz)}^2\Big)^{1/2}\\
&= C(q)\max\{|I|,1\})\Big(\sum _{k\in F}\mu(k,h)^{\min\{2s/k,2\}}\sum_{j\geq 0} 2^{2js}\|P_j \varphi\|^2_{l^2(h\zz)}\Big)^{1/2}\\
&\leq C(q,F)\max\{|I|,1\}\eps(s,h) \|\varphi\|_{B^s_{2,2}(\rr)}.
\end{align*}


We now prove \eqref{est.100} by showing that
\begin{equation}\label{est.j}
\min\{1,\mu 2^{jk}\}\leq \mu^{\min\{s/k,1\}}2^{js}
\end{equation} holds for all $\mu\geq 0$ and $j\geq 1$.
 It is obvious when $\mu\geq 1$. It remains to prove it in the case $\mu\leq 1$.
For any $|\xi|\geq 1$ we  have the following inequalities
\begin{align*}
\min\{1,\mu|\xi|^k\}&\leq \min\{1,\mu|\xi|^k\}^{\min\{s/k,1\}}=\min\{1,\mu^{\min\{s/k,1\}}|\xi|^{k\min\{s/k,1\}}\}\\
&\leq \mu^{\min\{s/k,1\}}|\xi|^{k\min\{s/k,1\}}\leq \mu^{\min\{s/k,1\}}|\xi|^s.
\end{align*}
Applying this inequality to $\xi=2^j$, $j\geq 0$, we get \eqref{est.j} and thus \eqref{est.100}. The proof of the
first step is now complete.


\medskip
{\bf Step II. Proof of  \eqref{est.201}.} Let us denote by $\Lambda_h$ the following operator:
$$\Lambda_h f(t)=\int_0^t  S_{A_h}(t-\sg)f(\sg)d\sg-\int    _{0}^tS_{B_h}(t-\sg)f(\sg)d\sg.$$
As in the case of the  homogenous estimate \eqref{est.200}, we  use a Paley-Littlewood decomposition of the
function $f$.
Inequality \eqref{paley-direct} and the fact that $\Lambda_h$ commutes with each projection $P_j$ give us
\begin{equation}\label{est.dir.paley}
    \|\Lambda_h f\|_{L^q(I,\,l^r(h\zz))}^2\leq c(q) \sum_{j\geq 0}\| P_j (\Lambda_h f)\|_{L^q(I,\,l^r(h\zz))}^2=
     c(q) \sum_{j\geq 0}\|\Lambda_h  (P_j f)\|_{L^q(I,\,l^r(h\zz))}^2.
\end{equation}
We claim that each term $\Lambda(P_j f)$ in the right hand side of  \eqref{est.dir.paley} satisfies:
\begin{align}\label{est.proj.inhom}
\|\Lambda_h &(P_j f)\|_{L^q(I,\,l^r(h\zz))}\\
\nonumber&\leq c(q,\tq)\max\{1,|I|\} \min\left\{\|P_j f\|_{L^{\tq'}(I,\,l^{\tr'}(h\zz))},\sum _{k\in F}\mu(k,h)
\||\nabla|^k P_j f\|_{L^{\tq'}(I,\,l^{\tr'}(h\zz))}\right\}.
\end{align}
In view of \eqref{est.j}, the above claim implies
\begin{align}\label{est.inmog.proj}
\|\Lambda_h &(P_j f)\|_{L^q(I,\,l^r(h\zz))}\\
\nonumber&\leq c(q,\tq)\max\{1,|I|\}\min\left\{\|P_j f\|_{L^{\tq'}(I,\,l^{\tr'}(h\zz))},\sum _{k\in F}\mu(k,h) 2^{jk}\| P_j f\|_{L^{\tq'}(I,\,l^{\tr'}(h\zz))}\right\}\\
\nonumber&=c(q,\tq)\max\{1,|I|\}\|P_j f\|_{L^{\tq'}(I,\,l^{\tr'}(h\zz))}\sum _{k\in F}\min\{1,\mu(k,h) 2^{jk}\}\\
\nonumber&\leq c(q,\tq)\max\{1,|I|\}\|P_j f\|_{L^{\tq'}(I,\,l^{\tr'}(h\zz))}\sum _{k\in F} \mu(k,h)^{\min\{s/k,1\}}2^{js}\\
\nonumber&\leq c(q,\tq)\max\{1,|I|\} \eps(s,h)2^{js}\|P_j f\|_{L^{\tq'}(I,\,l^{\tr'}(h\zz))}.
\end{align}
Estimates \eqref{est.dir.paley} and \eqref{est.inmog.proj} give us
\begin{equation}\label{laba.1}
 \|\Lambda_h f\|_{L^q(I,\,l^r(h\zz))}    \leq  c(q,\tq)\max\{1,|I|\}\eps(s,h)\Big(\sum_{j\geq 0} 2^{2js}
 \|P_j f\|_{L^{\tq'}(I,\,l^{\tr'}(h\zz))}^2  \Big)^{1/2}.
\end{equation}
Using that $\tq'\leq 2$, we can use the reverse Minkowski's inequality in $L^{\tilde q'/2}(I)$ to get

\begin{align*}\sum_{j\geq 0} 2^{2js}\|P_j f\|_{L^{\tq'}(I,\,l^{\tr'}(h\zz))}^2 & =
\sum_{j\geq 0} \big\| 2^{2js}\|P_j f\|^2_{l^{\tr'}(h\zz))}\big\|_{L^{\tilde q'/2}(I)}
\leq\Big\|  \sum_{j\geq 0} 2^{2js}\|P_j f\|^2_{l^{\tr'}(h\zz))}\Big\|_{L^{\tilde q'/2}(I)}
\\
&\lesssim
\Big\| \Big( \sum _{j\geq 0} 2^{2js} \|P_j f\|_{l^{\tr'}(h\zz)}^2 \Big) ^{1/2}\Big\|^2_{L^{\tq'}(I)}=
\|f\|^2_{L^{\tq'}(I,\,B^{s}_{\tr,2}(h\zz))}.
\end{align*}
By \eqref{laba.1} we get
$$\|\Lambda_h f\|_{L^q(I,\,l^r(h\zz))}    \leq c(q,\tq)\max\{1,|I|\}\eps(s,h) \|f\|_{L^{q'}(I,\,B^{s}_{\tr,2})}$$
which finishes the proof.

In the following we prove \eqref{est.proj.inhom}. Using that both operators $S_{A_h}$ and
$S_{B_h}$ fulfill  uniform  Strichartz estimates,
it is sufficient to prove that, under hypothesis \eqref{hyp.0}, the following estimate holds for all functions $f\in L^{\tq'}(I,\,l^{\tr'}(h\zz))$:
\begin{equation}\label{claim.304}
\|\Lambda_h  f\|_{L^q(I,\,l^r(h\zz))}\leq c(q,\tq)|I|\sum _{k\in F}a(k,h)
\||\nabla|^k  f\|_{L^{\tq'}(I,\,l^{\tr'}(h\zz))}.
\end{equation}
We point out that, in general, this estimate is not a direct consequence of \eqref{dis.stric.inhomg.2} since, under  assumption
\eqref{hyp.0}, we cannot establish the following inequality (of course, in the particular case $\tr'=2$ this can be obtained by Plancherel's identity)
$$\|(A_h-B_h)f\|_{L^{\tq'}(I,\,l^{\tr'}(h\zz))}\lesssim \sum _{k\in F}a(k,h)\||\nabla |^k f\|_{L^{\tq'}(I,\,l^{\tr'}(h\zz))}.$$
Identity \eqref{identity}  gives us that
$$\Lambda_h f(t)=
\int _0^t S_{A_h}(t-s)\Lambda _{1h} (A_h-B_h) f(s)ds$$
where
$$\Lambda _{1h}g(t)=\int_0^t    S_{B_h}(t-\sg)g(\sg)d\sg.$$
The inhomogeneous estimate \eqref{dis.stric.inhomg} with $(\tq',\tr')=(1,2)$ shows that
\begin{equation}\label{laba.3}
\|\Lambda_h f\|_{L^q(I,\,l^r(h\zz))}\leq c(q)\left\|\Lambda _{1h} (A_h-B_h)f\right\|_{L^1(I,\,l^2(h\zz))}.
\end{equation}
Using that ${B_h}$ satisfies
$S_{B_h}(t-\sg)=S_{B_h}(t)S_{B_h}(-\sg)=S_{B_h}(t)S_{B_h}(\sg)^*$ and that it commutes with $A_h$
we get
$$\Lambda _{1h} (A_h-B_h)f(t)=S_{B_h}(t)({A_h}-{B_h})\int_0^t S_{B_h}(\sg)^* f(\sg)d\sg.$$
Thus, using the uniform stability property, with respect to $h$, of the operators $S_{B_h}$:
 $$\|S_{B_h}(\cdot)\|_{l^2{(h\zz)}\rightarrow \,l^2(h\zz)}\lesssim 1$$ and hypothesis \eqref{hyp.0} we get
\begin{align}
\label{est.300}\|\Lambda _{1h} (A_h-B_h) f\|_{L^1(I,\,l^2(h\zz))}&\leq \Big\|({A_h}- {B_h})\int_0^t S_{B_h}(\sg)^*f(\sg)d\sg\Big\|_{L^1(I,\,l^2(h\zz))}\\
\nonumber&\leq \sum_{k\in F}a(k,h)\left\||\nabla|^k\int_0^t S_{B_h}(\sg)^* f(\sg)d\sg\right\|_{L^1(I,\,l^2(h\zz))}.
\end{align}
Using that ${B_h}$ and $|\nabla|$ commute, estimate \eqref{filsch401} with $U(\cdot)=S_{B_h}(\cdot)$ gives us that
\begin{align}
\nonumber\|\Lambda _{1h} (A_h-B_h) f\|_{L^1(I,\,l^2(h\zz))} &\leq |I|\sum_{k\in F}a(k,h)\left\|\int_0^t S_{B_h}(s)^* |\nabla|^kf(\sg)d\sg\right\|_{L^\infty(I,\,l^2(h\zz))}\\
\nonumber&\leq |I|\sum_{k\in F}a(k,h)\sup_{J\subset I}\left\|\int_{J} S_{B_h}(\sg)^* |\nabla|^kf(\sg)d\sg\right\|_{l^2(h\zz)}\\
\nonumber&\leq c(\tq)|I|\sum_{k\in F}a(k,h) \||\nabla|^k f\|_{L^{\tq'}(I,\,l^{\tr'}(h\zz)).}
\end{align}
Thus, by \eqref{laba.3} we obtain \eqref{claim.304} which finishes the proof.

\end{proof}

\section{Dispersive schemes for the linear Schr\"odinger equation}\label{dis.lse}
\setcounter{equation}{0}
In this section we  obtain error estimates for the numerical approximations of the linear Schr\"odinger equation.
We do this not only  in the $l^2(h\zz)$-norm but also in the auxiliary spaces that are needed  in the
analysis of the nonlinear Schr\"odinger equation.

\subsection{A general result.}
The numerical schemes we shall consider can all be
written in the abstract form
\begin{equation}\label{eq.dis}
    \left\{
\begin{array}{cc}
 \displaystyle i u^h_t(t) + A_h u^h=0, & t>0,  \\[8pt]
  \displaystyle u^h(0) =\th \varphi.
\end{array}
    \right.
\end{equation}
We assume that the operator $A_h$ is an approximation of the $1-d$ Laplacian. On the other hand, $\th \varphi$ is
an approximation of the initial data $\varphi$, $\th$ being a map from $L^2(\rr)$ into $l^2(h\zz)$  defined as follows:
\begin{equation}\label{def.th}
    (\th \varphi)(jh)=\int _{-\pi/h}^{\pi/h} e^{ijh \xi } \hat\varphi(\xi)d\xi.
\end{equation}
Observe that this operator acts by truncating the continuous Fourier transform of $\varphi$ on the interval
$(-\pi/h,\pi/h)$ and then    considering the discrete inverse Fourier transform on the
grid points $h\zz$. 

 To estimate the error committed  in the approximation of the LSE we  assume that the operator $A_h$, approximating the continuous Laplacian,
 has a symbol $a_h$ which satisfies
\begin{equation}\label{err.symbol}
    |a_h(\xi)-\xi^2|\leq \sum _{k\in F}a(k,h)|\xi|^k, \, \xi\in \left [-\frac\pi h,\frac\pi h\right ],
\end{equation}
for a finite set of indexes $F$.
As we shall see, different approximation schemes enter in this class for
different sets $F$ and orders $k$.

This condition on the operator $A_h$ suffices to analyze the rate of convergence
 in the $L^\infty(-T,T;\,l^2(h\zz))$ norm. However, one of our main objectives in this paper is to analyze this
error in the auxiliary norms $L^q(-T,T;\,l^r(h\zz))$ which is necessary for addressing the NSE with rough initial data.
More precisely, we need to identify classes of approximating operators $A_h$ of  the $1-d$ Laplacian so that the semi-discrete
semigroup $\exp(itA_h)$ maps uniformly, with respect to parameter $h$, $l^2(h\zz)$ into those spaces.

In the following we  consider operators $A_h$ generating dispersive schemes which are $l^2(h\zz)$-stable
\begin{equation}\label{l2-stable-h}
\|\exp(it A_h)\varphi\|_{l^2(h\zz)}\leq C\|\varphi\|_{l^2(h\zz)}, \ \forall\, t\geq 0
\end{equation}
and  satisfy the uniform $l^1(h\zz)-l^\infty(h\zz)$ dispersive property:
\begin{equation}\label{linfty.dis-h}
   \|\exp(it A_h)\varphi\|_{l^\infty(h\zz)}\leq \frac{C}{|t|^{1/2}}\|\varphi\|_{l^1(h\zz)}, \ \forall\, t\geq 0,
\end{equation}
for all $h>0$ and for all $\varphi\in l^1(h\zz)$, where the above constant $C$ is independent of $h$.
We point out that \eqref{l2-stable-h} is the standard stability property while the second one, \eqref{linfty.dis-h}, holds
only for well chosen numerical schemes.

Applying Theorem \ref{error} to the operator $B_h$ whose symbol is $-i\xi ^2$ and to $iA_h$, $A_h$ being the approximation
of the Laplace operator with the symbol $a_h(\xi)$,
we obtain the following result.

\begin{teo}\label{est.lin}
Let $s\geq 0$, $A_h$ satisfying  \eqref{err.symbol}, \eqref{l2-stable-h}, \eqref{linfty.dis-h},  and $(q,r)$ and
 $(\tq,\tr)$ be two $1/2$-admissible pairs. Denoting
\begin{equation}\label{def-eps}
\eps(s,h)=\sum _{k\in F} a(k,h)^{\min\{s/k,1\}},
\end{equation}
 the following  hold:

a) There exists a positive constant $C(q)$ such that
\begin{equation}\label{rez.princ}
\|\exp(itA_h)\th\varphi-\th \exp(it\partial_x^2)\varphi\|_{L^q(0,T; \,l^r(h\zz))}\leq \max\{1,T\}C(q)
\eps(s,h) \|\varphi\|_{H^s(\rr)}
\end{equation}
holds for all $\varphi\in H^s(\rr)$, $T>0$ and $h>0$.

b) There exists a positive constant $C(q,\tq)$ such that
\begin{align}\label{est.inh}
    \Big\|\int_0^t \exp(i(t-\sg)A_h)\th f(\sg)d\sg-&\int _0^t \th
    \exp(i(t-\sg)\partial_x^2)f(\sg)d\sg\Big\|_{L^q(0,T;\,l^r(h\zz))}\\
   \nonumber &\leq C(q,\tq)\max\{1,T\}\eps(s,h)\|f\|_{L^{\tq'}(0,T;\,   B^{s}_{\tr',2}(\rr))},
\end{align}
holds for all $T>0$,  $f\in L^{\tq'}(0,T;\, B^{s}_{\tr',2}(\rr))$ and $h>0$.

\end{teo}

\begin{remark}In the  particular case when $(q,r)=(\infty,2)$ and  the set $F$ of indices $k$ entering in the definition \eqref{def-eps} of
$\eps(s,h)$ is reduced to a simple element, the statements in this Theorem are proved in \cite{MR1005330} (Theorem 10.1.2, p. 201):
\begin{equation}\label{est.linf.l2}
\|\exp(itA_h)\th\varphi-\th \exp(it\partial_x^2)\varphi\|_{L^\infty(0,T; \,l^2(h\zz))}\leq C(q)T
\eps(s,h) \|\varphi\|_{H^s(\rr)}.
\end{equation}

%
\end{remark}


\begin{remark}
Observe that for $s\geq s_0=\max\{k:\,k\in F\}$ the function $s\rightarrow \eps(s,h)$ is independent of the $s$-variable:
$$\eps(s,k)=\eps(s_0,k)=\sum _{k\in F} a(k,h).$$ This means that imposing more  than $H^{s_0}(\rr)$ regularity on the initial
data does not improve the order of convergence in \eqref{rez.princ} and  \eqref{est.inh}.
\end{remark}

\begin{remark}In the case $0\leq s\leq s_0$, with $s_0$ as above, the  estimate $H^{s_0}(\rr)\rightarrow L^\infty(0,T;\,l^2(h\zz))$ in \eqref{rez.princ}
and the one given by the stability of the scheme
 $L^2(\rr)\rightarrow L^\infty(0,T;\,l^2(h\zz))$, allow  to obtain, using  an  interpolation argument, a weaker estimate:
$$\|\exp(itA_h)\th \varphi-\th \exp(it\partial_x^2)\varphi\|_{L^\infty(0,T;\,l^2(h\zz))}
    \leq C(T)
\eps(s_0,h)^{s/s_0}\|\varphi\|_{H^s(\rr)}.$$
If the set  $F$ has an unique element then this estimate is equivalent to  \eqref{rez.princ}.
However,  the improved estimates \eqref{rez.princ} and \eqref{est.inh} cannot be proved without using  Paley-Littlewood's decomposition, as in the proof of Theorem \ref{error}.
\end{remark}

\subsection{Examples of operators $A_h$.}\label{examples}
In this section we will analyze various operators $A_h$ which approximate the $1-d$ Laplace operator $\partial_x^2$.


\textbf{Example 1. The 3-point conservative approximation.}
The simplest example of approximation scheme for  the Laplace operator $\partial_x^2$ is given by
 the classical finite difference approximation  $\Delta_h$
\begin{equation}\label{discrete.lap}
(\Delta _h u)_j
=\frac{u_{j+1}+u_{j-1}-2u_j}{h^2}.
\end{equation}
It satisfies  hypothesis \eqref{err.symbol} with $F=\{4\}$ and $a(4,h)=h^2$. Thus,  we are dealing with an approximation scheme of order two. Indeed, we have:
$$\Big|\frac {4}{h^2}\sin^2\Big(\frac {\xi h}2\Big)-\xi^2\Big|\lesssim h^2|\xi|^4,\, \forall\,\xi \in \Big[-\frac \pi h,\frac \pi h\Big].$$
%
%
However, this operator does not satisfy  \eqref{linfty.dis-h} with a constant $C$ independent of the mesh size
 $h$, (see \cite{1063.35016}, Theorem 1.1) and Theorem \ref{est.lin} cannot be applied. This means that we cannot  obtain the same estimate as for  second order dispersive schemes:
 \begin{equation}\label{est.dis.1}
 \|\exp(itA_h)\th\varphi-\th \exp(it\partial_x^2)\varphi\|_{L^q(0,T; \,l^r(h\zz))}\leq C(q,T)
 \|\varphi\|_{H^s(\rr)}\left\{
 \begin{array}{ll}
h^{s/2},&s\in (0,4),\\
h^2,&s >4.
 \end{array}
 \right.
 \end{equation}

 However, using the ideas of Brenner on the order of convergence in the $l^r(h\zz)$-norm, $r>2$, (\cite{MR0461121}, Ch. 6, Theorem 3.2, Theorem 3.3 and Ch.3, Corollary 5.1) we can get the following estimates:
 \begin{align*}\|\exp(itA_h)\th\varphi-&\th \exp(it\partial_x^2)\varphi\|_{L^q(0,T; \,l^r(h\zz))}\\[10pt]
 &\leq C(q,T)
 \|\varphi\|_{B^s_{r,\infty}(\rr)}\left\{
 \begin{array}{ll}
h^{\frac 12(s-1+\frac 2r)},&s\in (0,4+1-\frac 2r),\\[10pt]
h^2,&s \geq 4+1-\frac 2r,
 \end{array}
 \right.\\[10pt]
 &\leq C(q,T)
 \|\varphi\|_{H^{s+\frac 12-\frac 1r}(\rr)}\left\{
 \begin{array}{ll}
h^{\frac 12(s-1+\frac 2r)},&s\in (0,4+1-\frac 2r),\\[10pt]
h^2,&s \geq 4+1-\frac 2r,
 \end{array}
 \right.
 \end{align*}
 where we have used that $H^{s_0}(\rr)=B^{s_0}_{2,2}(\rr) \hookrightarrow B^s_{r,\infty}(\rr)$ when $s_0-1/2=s-1/r$.

 Observe that in the case $s\in (0,4)$ the above estimate guarantees that
 \begin{align}\label{est.nodis.1}
 \|\exp(itA_h)\th\varphi-&\th \exp(it\partial_x^2)\varphi\|_{L^q(0,T; \,l^r(h\zz))}\\
\nonumber &\leq C(q,T)\|\varphi\|_{H^{s+\frac 12-\frac 1r}(\rr)} h^{\frac 12(s+\frac 12-\frac 1r)}h^{-\frac 32(\frac 12-\frac 1r)}.
 \end{align}
Moreover for any $\sigma \in (1/2-1/r,4+1/2-1/r)$ we can find $s\in (0,4)$ with $\sigma=s+1/2-1/r$ and
using \eqref{est.nodis.1} we obtain
\begin{equation}\label{lossfactor}
 \|\exp(itA_h)\th\varphi-\th \exp(it\partial_x^2)\varphi\|_{L^q(0,T; \,l^r(h\zz))}\leq C(q,T)\|\varphi\|_{H^{\sigma}(\rr)} h^{\frac \sigma 2}h^{-\frac 32(\frac 12-\frac 1r)}.
 \end{equation}
 In the case of an approximation of order two one could expect  the error in the above estimate to be of order
  $h^{ \sigma/ 2}$ as in the $L^\infty(0,T;\, l^2(h\zz))$ case.  But, here we get an extra factor of order    $h^{-3/2( 1/2- 1/r)}$ which diverges unless $r=2$, which corresponds to the classical energy estimate in $L^\infty(0,T;\,L^2(\rr))$.
This does not happen in the case of a second order dispersive approximation of the Schr\"odinger operator, where Theorem \ref{est.lin} give us an order of error as in \eqref{est.dis.1}.

Note that, according to Theorem \ref{est.lin}, this loss in the rate of convergence is due to the lack of dispersive properties of the scheme.

Also we point out that to obtain an error of order $h^2$ in \eqref{est.nodis.1} we need to consider initial data in $H^{4+1-2/r}(\rr)$.  So we need to impose 
an extra regularity condition of $1-2/r$ derivatives on the initial data $\varphi$ to assure the same order of convergence as the one in  \eqref{est.dis.1} for dispersive schemes.

\textbf{Example 2. Fourier filtering of the 3-point conservative approximation.}
Another example is given by the spectral filtering
$\Delta_{h,\gamma}$ defined by:
\begin{equation}\label{spectral.delta}
    \Delta_{h,\gamma} \varphi =\Delta_h ({\bf 1}_{(-\frac {\gamma \pi} h,\frac {\gamma \pi} h)}\hat\varphi)^\vee, \, \gamma<\frac 1 2.
\end{equation}
In other words, $ \Delta_{h,\gamma} $ is a discrete operator whose action is as follows:
$$(\Delta_{h,\gamma} \varphi)_j=\int _{-\gamma \pi/h}^{\gamma \pi/h}\frac{4}{h^2}\sin^2\big(\frac {\xi h}2\big) e^{ijh\xi}\hat{\varphi}(\xi)d\xi,\, j\in \zz,$$
i.e. it has the symbol $$a_{h,\gamma}(\xi)= \frac{4}{h^2}\sin^2\big(\frac {\xi h}2\big){\bf 1}_{(-\gamma \pi/h,\gamma \pi/h)}. $$

In this case
$$|a_{h,\gamma}(\xi)-\xi^2|\leq c(\gamma)\left\{
\begin{array}{cc}
  h^2\xi ^4, & |\xi|\leq \pi \gamma/h, \\[10pt]
  \xi^2, &  |\xi|\geq \pi \gamma/h
\end{array}
\right.
\leq c(\gamma)h^2\xi^4\quad \text{for all}\quad\xi \in \Big[-\frac \pi h,\frac \pi h\Big].$$
Thus  $\Delta_{h,\gamma}$ constitutes an approximation of the Laplace operator $\Delta$ of order two and 
 the semigroup generated by $i\Delta _{h,\gamma}$ has uniform dispersive properties (see \cite{liv3}).
 Theorem
\ref{est.lin}, which exploits the dispersive character of the numerical scheme, gives us
 $$\|\exp(itA_h)\th\varphi-\th \exp(it\Delta)\varphi\|_{L^q(0,T; \,l^r(h\zz))}\leq C(q,T)
 \|\varphi\|_{H^s(\rr)}\left\{
 \begin{array}{ll}
h^{s/2},&s\in (0,4),\\
h^2,&s >4.
 \end{array}
 \right.
 $$
We note that using the same arguments based on  $l^r(h\zz)$-error estimates (given in \cite{MR0461121}),  as in the Example 1, we can obtain the same result only if $r=2$ or
assuming more regularity of the initial data $\varphi$.

This scheme, however, has a serious drawback to be implemented in nonlinear problems since it requires the Fourier filtering  to be applied on the initial data and also on the nonlinearity, which is computationally expensive.

%

\medskip
\textbf{Example 3. Viscous approximation.}
To overcome the lack of uniform $L^q(I,l^r(h\zz))$ estimates, in \cite{liv3} and \cite{liv.as.an} numerical schemes based in adding extra
numerical viscosity have been introduced.
The first possibility is to take $A_h=\Delta_h+ia(h)\Delta_h$  with $a(h)=h^{2-1/\alpha(h)}$ and $\alpha(h)\rightarrow 1/2$
such that $a(h)\rightarrow 0$. In this case \eqref{err.symbol} is satisfied as follows:
\begin{equation}\label{err.diss}
\Big|\frac{4}{h^2}\sin^2\big(\frac {\xi h}2\big)+ia(h)\frac{4}{h^2}\sin^2\big(\frac {\xi h}2\big)-\xi^2\Big|
\leq h^2\xi^4+a(h)\xi^2.
\end{equation}
This numerical approximation of the Schr\"odinger semigroup has been used in \cite{liv3} and \cite{liv-siam} to construct convergent
numerical schemes for the NSE. However, the special choice of  the function $a(h)$ that is required,  shows  that the error in the right hand side of \eqref{err.diss}
 goes to zero slower that any polynomial function of $h$ and thus, at least theoretically, the
 convergence towards LSE, and, consequently  to the NSE, will be very slow. Thus, we will not further analyze this scheme.

\medskip \textbf{Example 4. A higher order viscous approximation.}
A possibility to overcome the  drawbacks of the previous scheme, associated to
 the different behavior of the
$l^1(h\zz)-l^\infty(h\zz)$ decay rate of the solutions,  is
to choose higher order dissipative schemes as introduced in \cite{liv.as.an}:
\begin{equation}\label{diss.aprox}
    A_h=\Delta_h-ih^{2(m-1)}(-\Delta_h)^m, \, m\geq 2.
\end{equation}
In this case, hypothesis \eqref{err.symbol} reads:
\begin{equation}\label{est.diss}
\Big|\frac{4}{h^2}\sin^2\big(\frac {\xi h}2\big)+i h^{2(m-1)}\Big(\frac{4}{h^2}\sin^2\big(\frac {\xi h}2\big)\Big)^{m}-\xi^2\Big|
\leq h^2\xi^4+h^{2(m-1)}\xi^{2m}.
\end{equation}
Theorem \ref{est.lin} then guarantees  that for any $0\leq s\leq 4$ the following estimate holds:
\begin{align*}
\|\exp(it A_h)\th\varphi-\th \exp(it\Delta)\varphi\|_{L^q(0,T; \,l^r(h\zz))}&\leq
\max\{1,T\}(h^{s/2}+h^{(m-1)s/m})\|\varphi\|_{H^s(\rr)}\\
&\leq \max\{1,T\}  h^{s/2} \|\varphi\|_{H^s(\rr)}.
\end{align*}
Thus we obtain the same order of error as for the discrete Laplacian $A_h=\Delta_h$ but this time not only in the
$L^\infty(I;\, l^2(h\zz))$-norm but in all the auxiliary  $L^q(I,\,l^r(h\zz))$-norms.
We thus get the same optimal results as for the other dispersive scheme in Example 2 based on Fourier filtering.

\section{A two-grid algorithm}\label{two.grid.section}
\setcounter{equation}{0}
In this section we analyze one further 
 strategy introduced in  \cite{liv2}, \cite{liv3} to recover the uniformity of the dispersive properties.
It   is based on the two-grid algorithm that we now describe. We consider the standard conservative 3-point
approximation of the laplacian: $A_h=\Delta_h$.
 But, this time, in order to avoid the lack
of dispersive properties associated with the high frequency
components, the scheme will be restricted to the class
of slowly oscillatory data obtained by a two-grid algorithm. The main advantage of
this filtering method with respect to the Fourier one is that the
filtering can be realized in the physical space.

The method, inspired by \cite{0763.76042}, is roughly as
follows. We consider two meshes: the coarse one of size $4h$, $h>0$,
$4h\zz$, and the finer one, the computational one, $h\zz$, of
size $h>0$. The method relies basically on solving the
finite-difference semi-discretization  on the fine mesh
$h\zz$, but only for slowly oscillating  data, interpolated from
the coarse grid $4h\zz$. The $1/4$ ratio between
the two meshes is important to guarantee the dispersive properties of the
method. This particular structure of the data cancels the
pathology of the discrete symbol at the points $\pm \pi/2h$.

To be more precise  we  introduce the extension operator
$\bpi ^{4h}_h $ which associates to any function $\psi:4h\zz\rightarrow \cn$  a new function
$\bpi ^{4h}_h\psi:h\zz\rightarrow \cn$ obtained by an interpolation process:
$$(\bpi ^{4h}_h \psi)_j=({\bf P}_{4h}^1 \psi) (j h),\, j\in \zz,$$
where ${\bf P}_{4h}^1 \psi$ is the piecewise linear
interpolator of $\psi$.

 The semi-discrete method we propose is the
following:
\begin{equation}\label{two-grid}
    \left\{
\begin{array}{cc}
 \displaystyle i u^h_t(t) + \Delta_h u^h=0, & t>0,  \\
 \\
  \displaystyle u^h(0) =\bpi^{4h}_h \tph \varphi. &
\end{array}
    \right.
\end{equation}

The Fourier transform of the two-grid initial datum can be characterized as follows
(see Lemma 5.2, \cite{liv3}):
\begin{equation}\label{sym.twogrid}
(\bpi ^{4h}_h \tph \varphi)^\wedge(\xi)=m(h\xi )\widetilde{\tph\varphi}(\xi),\,\xi \in \left[-\frac \pi h,\frac\pi h\right],
\end{equation}
where $\widetilde{\tph\varphi}(\xi)$ is the extension by periodicity of the function $\widehat{\tph\varphi}$,
initially defined on $[-\pi/4h,\pi/4h]$, to the interval $[-\pi/h,\pi/h]$,
and
\begin{equation}\label{two.grid.mul}
m(\xi)=\left(\frac{e^{4i\xi}-1}{4(e^{i\xi}-1)}\right)^2, p\geq 2.
\end{equation}

The following result, proved in \cite{liv2}, guarantees that system \eqref{two-grid} is dispersive in the sense that the
discrete version of the Strichartz inequalities hold, uniformly on $h>0$.
\begin{teo}\label{multstr}
Let  $(q,r)$, $(\tq,\tr)$ be two  $1/2$-admissible pairs.
The following properties hold


i) There exists a positive constant $C(q)$ such that
\begin{equation}\label{mest2}
    \|e^{it\Delta _h}\bpi^{4h}_h \varphi\|_{L^q(\rr,\,l^r(h\zz))}\leq C(q)
    \|\bpi^{4h}_h\varphi\|_{l^2(h\zz)}
\end{equation}
uniformly on $h>0$.

%

ii) There exists a positive constant $C(d,r,\tr)$ such that
\begin{equation}\label{mest4}
 \left \|\int _{s<t} e^{i(t-s)\Delta _h}\bpi^{4h}_h
f(s)ds\right\|_{L^q(\rr,\,l^r(h\zz))}\leq C(q,\tq)\|\bpi^{4h}_h
f\|_{L^{\tilde{q}'}(\rr,\,l^{\tilde{r}'}(h\zz))}
\end{equation}
 for all $f\in L^{\tilde{q}'}(\rr,\,l^{\tilde{r}'}(4h\zz))$, uniformly in $h>0$.
\end{teo}

In the following lemma we  estimate the error introduced by the two-grid algorithm.
\begin{teo}\label{est.two.grid}
Let $s\geq 0$ and $(q,r)$, $(\tq,\tr)$ be two admissible pairs.

a) There exists a positive constant $ C(q,s)$ such that
\begin{align}\label{tg.1}
    \|\exp(it\Delta_h)\bpi ^{4h}_h \tph \varphi-&\th \exp(it\partial_x^2)\varphi\|_{L^q(I;\,l^r(h\zz))}\\[10pt]
   \nonumber &\leq C(q,s)
\max\{1,|I|\}\big(h^{\min\{s/2,2\}}+h^{\min\{s,1\}}\big)\|\varphi\|_{H^s(\rr)},
\end{align}
holds for all $\varphi\in H^s(\rr)$ and $h>0$.

b) There exists a positive constant $C(q,\tq,s)$ such that
\begin{align}\label{tg.inhom}
   \Big \|\int _{s<t}\exp(i(t-s)\Delta_h) &\bpi ^{4h}_h \tph f(s)ds-\int _{s<t} \th \exp(i(t-s)\partial_x^2)f(s)ds\Big\|_{L^q(I;\,l^r(h\zz))}\\[10pt]
\nonumber &\leq C(q,\tq,s) \max\{1,|I|\}\big(h^{\min\{s/2,2\}}+h^{\min\{s,1\}}\big)\|f\|_{L^{\tq'}(I;\,B^{s}_{\tr',2}(\rr))}.
\end{align}

\end{teo}

\begin{remark} 
There are two error terms in the above estimates: $h^{\min\{s/2,2\}}$ and $h^{\min\{s,1\}}$. The first one comes from a second order numerical scheme generated by the approximation of the laplacian $\partial_x^2$ with $\Delta_h$  and the second one from the use of a two-grid interpolator. Observe that for  initial data
$\varphi\in H^s(\rr)$, $s\in (0,2)$ the results are the same as in the case of the second order schemes.
Also, imposing more than $H^2(\rr)$ regularity on the initial data does not improve the order of 
convergence. This is a consequence of the fact that the two-grid interpolator appears.
The multiplier $m(\xi)$  defined in \eqref{two.grid.mul}  satisfies $m(\xi)-1\simeq \xi$ as $\xi \sim 0$ and then the following estimate, which occurs in the proof of Theorem \ref{est.two.grid},
$$\int_{-\pi/4h}^{\pi/4h}|m(h\xi)-1|^2|\hat\varphi(\xi)|^2d\xi
\lesssim (h\|\varphi\|_{H^1(\rr)})^2,$$
cannot be improved by imposing more regularity on the function $\varphi$.
\end{remark}

\begin{proof}[Proof of Lemma \ref{est.two.grid}]
{\textbf{Case I. Proof of the homogenous estimate \eqref{tg.1}}.}
Let us consider $\Delta ^h$ acting on discrete functions as follows:
$$(\Delta^h\varphi)_j=\int _{-\pi/h}^{\pi/h}\xi^2 e^{ij\xi h}\hat\varphi(\xi)d\xi.$$
Note that $\Delta^h$ differs from the finite-difference approximation $\Delta_h$ on the fact that, in $\Delta^h$, $\xi^2$ replaces the symbol $4/h^2 \sin^2(\xi h/2)$ of $\Delta_h$.

In view of the definition of $\Delta^h$, we have $$\exp(it\Delta^h)\th \varphi=\th \exp(it\partial_x^2)\varphi.$$
Using the last identity, we write  
\begin{align*}
\exp(it\Delta_h)\bpi ^{4h}_h \th\varphi-\th \exp(it\partial_x^2)\varphi&=
\exp(it\Delta_h)\bpi ^{4h}_h \th\varphi-\exp(it\Delta^h)\th \varphi\\
&=I_1(t)+I_2(t)
\end{align*}
 where:
$$I_1(t)=\exp(it\Delta_h)\bpi ^{4h}_h \tph \varphi-\exp(it\Delta^h)\bpi ^{4h}_h \tph \varphi$$
and
$$I_2(t)=\exp(it\Delta^h)\bpi ^{4h}_h \tph \varphi-\exp(it\Delta^h)\th \varphi.$$ In the following we estimate each of them.

Applying Theorem \ref{error} to operators $\Delta_h$ and $\Delta^h$ we
get
\begin{align*}
\|I_1\|_{L^q(0,T;\,l^r(h\zz))}\leq h^{\min\{s/2,2\}} \max\{1,T\}\|\bpi ^{4h}_h \tph \varphi \|_{B^s_{2,2}(h\zz)}\leq
h^{\min\{s/2,2\}} \max\{1,T\}\|\varphi \|_{H^s(\rr)}.
\end{align*}


In the case of $I_2$ we claim that for any $s\geq 0$
\begin{equation}\label{est.i20}
\|I_2\|_{L^q(0,T;\,l^r(h\zz))}\leq h^{\min\{s,1\}}\|\varphi\|_{H^s(\rr)}.
\end{equation}

To prove this claim,  we remark that the operator $\exp(it\Delta^h)$ satisfies \eqref{l2-stable} and \eqref{linfty.dis}. Thus  Proposition \ref{strfil} guarantees that  $\exp(it\Delta^h)$ has 
 uniform Strichartz  estimates and
\begin{equation}\label{est.i2}
\|I_2\|_{L^q(0,T;\,l^r(h\zz))}\leq \|\bpi^{4h}_h\tph\varphi-\th\varphi\|_{l^2(h\zz)}.
\end{equation}
It is then sufficient to prove that
\begin{equation}\label{est.interpolatori}
\|\bpi^{4h}_h\tph\varphi-\th\varphi\|_{l^2(h\zz)}\leq h^{\min\{s,1\}}\|\varphi\|_{H^s(\rr)}
\end{equation}
holds for any $s\geq 0$. Actually it suffices to prove it for $0\leq s \leq 1$.
 Also the cases $s\in (0,1)$ follow by intepolation between the cases $s=0$ and $s=1$. 
We will consider now these two cases.

The case $s=0$ easily follows since 
$$\||\bpi^{4h}_h\tph\varphi\|_{l^2(h\zz)}\lesssim \|\tph\varphi\|_{l^2(4h\zz)}\lesssim \|\varphi\|_{L^2(\rr)}$$
and
$$\|\th\varphi\|_{l^2(h\zz)}\lesssim \|\varphi\|_{L^2(\rr)}.$$

We now prove \eqref{est.interpolatori} in the case $s=1$: 
\begin{equation}\label{case.s1}
\|\bpi^{4h}_h\tph\varphi-\th\varphi\|_{l^2(h\zz)}\lesssim h\|\varphi\|_{H^1(\rr)}.
\end{equation}
Using that 
$$\|\tph\varphi-\th\varphi\|_{l^2(h\zz)}\leq \Big(\int _{|\xi|\geq \pi/4h}|\hat\varphi(\xi)|^2d\xi\Big)^{1/2}\lesssim h\|\varphi\|_{H^1(\rr)},$$
it is sufficient to prove the following estimate
\begin{equation}\label{est.tg.100}
    \|\bpi^{4h}_h\tph\varphi-\tph\varphi\|_{l^2(h\zz)}\lesssim h\|\varphi\|_{H^1(\rr)}.
\end{equation}
The representation formula \eqref{sym.twogrid}   gives us that
\begin{align}
\label{est.600}\|\bpi^{4h}_h\tph \varphi-\tph\varphi\|_{l^2(h\zz)}^2\leq &\int_{-\pi/4h}^{\pi/4h}|m(h\xi)-1|^2|\hat\varphi(\xi)|^2d\xi\\
\nonumber&+\int_{\pi/4h\leq |\xi|\leq \pi/h}|m(h\xi)|^2 |\widetilde{\tph \varphi}(\xi)|^2d\xi.
\end{align}
Using that $|m(\xi)-1|\leq |\xi|$ for $\xi\in [-\pi/4,\pi/4]$ we obtain
\begin{equation}\label{est.602}
\int_{-\pi/4h}^{\pi/4h}|m(h\xi)-1|^2|\hat\varphi(\xi)|^2d\xi
\lesssim (h\|\varphi\|_{H^1(\rr)})^2.
\end{equation}
Previous results on the Fourier analysis of the two-grid method (see \cite{liv-jems}, Appendix B) and
the periodicity with period $\pi/2h$ of the function $\widetilde{\tph \varphi}(\xi)$  give us that
\begin{align*}
\int_{\pi/4h\leq |\xi|\leq \pi/h}&|m(h\xi)|^2 |\widetilde{\tph \varphi}(\xi)|^2d\xi=
\int_{\pi/4h\leq |\xi|\leq \pi/h}\left|\frac{e^{4i\xi h}-1}{4(e^{i\xi h}-1)}\right|^4|\widetilde{\tph \varphi}(\xi)|^2\\
&\lesssim \int_{\pi/4h\leq |\xi|\leq \pi/h}|{e^{4i\xi h}-1}|^4|\widetilde{\tph \varphi}(\xi))|^2
\lesssim \int_{-\pi/4h\leq \xi\leq \pi/4h} |{e^{4i\xi h}-1}|^4 |\widetilde{\tph \varphi}(\xi)|^2\\
&\lesssim \int_{-\pi/4h\leq \xi\leq \pi/4h} |\xi h|^4 |\widetilde{\tph \varphi}(\xi)|^2d\xi\lesssim
(h\|\varphi\|_{H^1(\rr)})^2.
\end{align*}

We obtain that \eqref{est.tg.100} holds and, consequently,  \eqref{case.s1} too. Thus \eqref{est.i20} is satisfied for any
positive $s$.

Observe that the   main term in the right
hand side of \eqref{est.600} is given by \eqref{est.602}, and this estimate  cannot be improved by imposing more than
$H^1(\rr)$ smoothness on $\varphi$.


\medskip
{\textbf{Case II. Proof of the inhomogeneous estimate \eqref{tg.inhom}}.}
We proceed as in the previous case by splitting the
difference we want to evaluate as
$$\int _{s<t}\exp(i(t-s)\Delta_h) \bpi ^{4h}_h \tph f(s)ds-\int _{s<t} \th \exp(i(t-s)\partial_x^2)f(s)ds=I_1+I_2$$
where
$$I_1=\int _{s<t} \big(\exp(i(t-s)\Delta_h) -\exp(i(t-s)\Delta^h)\big)\bpi ^{4h}_h\tph f(s)ds,$$
and
$$I_2=\int _{s<t} \exp(i(t-s)\Delta^h)(\bpi ^{4h}_h\tph f(s)-\th f(s))ds.$$

 In the case of $I_1$, applying Theorem \ref{error} to operators $\Delta_h$ and $\Delta^h$, we get
$$\|I_1\|_{L^q(0,T;l^r(h\zz))}\leq  h^{\min\{s/2,2\}}\max\{1,T\}\|\bpi ^{4h}_h\tph f\|_{L^{\tq'}(0,T;\,B^s_{\tr',2}(h\zz))}.$$
Applying Theorem \ref{graf} below to the multiplier $m$ given by \eqref{two.grid.mul}, for any $s>0$ we obtain that
$$\|\bpi ^{4h}_h\tph f\|_{L^{\tq'}(0,T;\,B^s_{\tr',2}(h\zz))}\leq \| f\|_{L^{\tq'}(0,T;\,B^s_{\tr',2}(\rr))}$$
and then $I_1$ satisfies:
\begin{equation}\label{est.inh.i1}
\|I_1\|_{L^q(0,T;l^r(h\zz))}\leq  h^{\min\{s/2,2\}}\max\{1,T\}\| f\|_{L^{\tq'}(0,T;\,B^s_{\tr',2}(\rr))}.
\end{equation}

In the case of $I_2$ we claim that
\begin{equation}\label{est.inh.i2}
\|I_2\|_{L^q(0,T;l^r(h\zz))}\leq h^{\min\{s,1\}}\| f\|_{L^{\tq'}(0,T;\,B^s_{\tr',2}(\rr))}.
\end{equation}
To prove this claim we consider the  cases $s=0$ and $s=1$.  When $s\in (0,1)$ we use interpolation between the previous ones. Also the case $s>1$ follows by using the embedding 
$B^s_{\tr',2}(\rr) \hookrightarrow B^1_{\tr',2}(\rr)$.

The case $s=0$ follows from  Proposition \ref{strfil} applied to the operators $U_h(t)=\th \exp(it \partial_x^2)$.

 We now consider the case $s=1$. Using Strichartz estimates given by Proposition \ref{strfil} to the operator $\exp(it\Delta^h)$ we get:
$$\|I_2\|_{L^q(0,T;l^r(h\zz))}\leq \|\bpi ^{4h}_h \tph f-\th f\|_{L^{\tq'}(0,T;\,l^{\tr'}(h\zz))}.$$
Theorem \ref{graf} applied to the multiplier $m$ gives us
$$\|\bpi ^{4h}_h \tph f-\tph f\|_{L^{\tq'}(0,T;\,l^{\tr'}(h\zz))}\leq h \| f\|_{L^{\tq'}(0,T;\,B^1_{\tr',2}(\rr))}$$
and
$$\|\tph f-\th f\|_{L^{\tq'}(0,T;\,l^{\tr'}(h\zz))} \leq h \| f\|_{L^{\tq'}(0,T;\,B^1_{\tr',2}(\rr))}.$$
Thus \eqref{est.inh.i2} holds for $s=1$, and in view of the above comments, for all $s\geq 0$.

Putting together \eqref{est.inh.i1} and \eqref{est.inh.i1} we obtain the inhomogeneous estimate \eqref{tg.inhom}.

 The proof is now complete.
\end{proof}

\section{Convergence of the dispersive method for the NSE}\label{non}
\setcounter{equation}{0}
In this section we  introduce  numerical schemes for the NSE based on dispersive approximations of the LSE. We  first
present some classical results on well-posedness and regularity of solutions of the NSE. Secondly we  obtain the order
of convergence for the approximations of the NSE described above.

\subsection{Classical facts on NSE}
We  consider the NSE with nonlinearity $f(u)=|u|^pu$ and $\varphi\in H^s(\rr)$. We are interested in the
case of $H^s(\rr)$ initial data with $s\leq 1$. The following well-posedness result is known.

\begin{teo}\label{exis.uniq}
Let $f(u)=|u|^pu$ with $p\in (0,4)$. Then

i) (Global existence and uniqueness, \cite{1055.35003}, Th. 4.6.1, Ch.~4, p.~109)\\ For any $\varphi\in L^2(\rr)$, there exists a unique global solution $u$
of \eqref{scha869} in the class
$$u\in C(\rr,L^2(\rr))\cap L^q_{loc}(\rr,L^r(\rr))$$
for all $1/2$-admissible pairs $(q,r)$ such that
$$\|u(t)\|_{L^2(\rr)}=\|\varphi\|_{L^2(\rr)},\, \forall \, t\in \rr.$$

ii) (Stability, \cite{1055.35003}, Th. 4.6.1, Ch.~4, p.~109) Let $\varphi$ and $\psi$ be two $L^2(\rr)$ functions, and    $u$ and $v$ the corresponding solutions of the NSE. Then
for any $T>0$ there exists a positive constant $C(T,\|\varphi\|_{L^2(\rr)},\|\psi\|_{L^2(\rr)})$ such that the following holds
\begin{equation}\label{est.nse.1}
    \|u-v\|_{L^\infty(0,T;\,L^2(\rr))}\leq C(T,\|\varphi\|_{L^2(\rr)},\|\psi\|_{L^2(\rr)})
\|\varphi-\psi\|_{L^2(\rr)}
\end{equation}

iii) (Regularity)
Moreover if $\varphi\in H^s(\rr)$, $s\in (0,1/2)$ then (\cite{1055.35003}, Theorem 5.1.1, Ch.~5, p.~147)
$$u\in C(\rr,H^s(\rr))\cap L^q_{loc}(\rr,B^s_{r,2}(\rr))$$for every admissible pairs
$(q,r)$.

Also  if $\varphi\in H^1(\rr)$ then $u\in C(\rr, H^1(\rr))$ (\cite{1055.35003}, Theorem 5.2.1, Ch.~5, p.~149).


\end{teo}


\begin{remark}
The embedding $B^{s}_{r,2}(\rr)\hookrightarrow W^{s,r}(\rr), r\geq 2$, (see \cite{1055.35003}, Remark 1.4.3, p. 14) guarantees that, in particular,
$u\in L^q_{loc}(\rr,W^{s,r}(\rr))$. Moreover, $f(u)\in L^{q'}_{loc}(\rr,B^{s}_{r',2}(\rr))$ and for any $0<s\leq 1$
(see \cite{1055.35003}, formula
(4.9.20), p.~128)
 \begin{equation}\label{normebs}
 \|f(u)\|_{L^{q'}(I,B^{s}_{r',2}(\rr))}\lesssim |I|^{\frac {4-p(1-2s)}4}\|u\|_{L^q(I,B^s_{r,2}(\rr))}^{p+1}.
 \end{equation}
\end{remark}

The fixed point argument  used to prove the existence and uniqueness result in Theorem \ref{exis.uniq}
gives us also quantitative information of the solutions of NSE in terms of the $L^2(\rr)$-norm
of the initial data. The following holds:

\begin{lemma}\label{stab.1}
Let $\varphi\in L^2(\rr)$ and $u$ be the solution of the NSE with initial data $\varphi$ and nonlinearity $f(u)=|u|^pu$, $p\in (0,4)$,
as in Theorem \ref{exis.uniq}. There exists $c(p)>0$ and $T_0=c(p)\|\varphi\|_{L^2(\rr)}^{-4p/(4-p)}$ such that
 for any $1/2$-admissible pairs $(q,r)$, there exists a positive constant $C(p,q)$ such that
\begin{equation}\label{est.nse}
    \|u\|_{L^{q}(I;\,L^r(\rr))}\leq C(p,q) \|\varphi\|_{L^2(\rr)}
\end{equation}
holds for all intervals $I$ with  $|I|\leq T_0$.
\end{lemma}



\begin{proof}[Proof of Lemma \ref{stab.1}]Let us fix an admissible pair $(q,r)$.
The fixed point argument used in the proof of Theorem \ref{exis.uniq} (see (\cite{MR559676}, Th.~5.5.1, p.~15)
gives us the existence of  a time
$T_0$,
$$T_0=c(p) \|\varphi\|_{L^2(\rr)}^{-\frac{4p}{4-p}},$$
such that
$$\|u\|_{L^q(0,T_0;\,L^r(\rr))}\leq C(p,q)\|\varphi\|_{L^2(\rr)}.$$
The same argument applied to the interval $[(k-1)T_0,kT_0]$, $k\geq 1$, and the  conservation of the $L^2(\rr)$-norm
of the solution $u$ of the NSE gives us that
$$\|u\|_{L^q((k-1)T_0,\,kT_0;\,L^r(\rr))}\leq C(p,q)\|u((k-1)T_0)\|_{L^2(\rr)}=C(p,q)\|\varphi\|_{L^2(\rr)}.$$
This proves \eqref{est.nse} and finishes the proof of Lemma \ref{stab.1}.
\end{proof}

\subsection{Approximation of the NSE by dispersive numerical schemes}

In this section we consider a numerical scheme for the NSE based on approximations of the LSE
that has uniform dispersive properties of
Strichartz type. Examples of such schemes have been given in Section \ref{dis.lse} and Section \ref{two.grid.section}.

To be more precise, we  deal with the following numerical schemes:
\begin{itemize}
\item Consider
\begin{equation}\label{eq.dis.non}
    \left\{
\begin{array}{cc}
 \displaystyle i u^h_t + A_h u^h=f(u^h), & t>0,  \\[8pt]
  \displaystyle u^h(0) =\varphi^h, &
\end{array}
    \right.
\end{equation}
where $A_h$ is an approximation of $\Delta$ such
that  $\exp(it A_h)$  has uniform dispersive properties of
Strichartz type. We also assume that 
 $A_h$ satisfies  ${\rm{Re}}(iA_h\varphi, \varphi)_h\leq 0$, $\Re$ being the real part,  and has a symbol $a_h(\xi)$ which verifies
\begin{equation}\label{err.symbol.1}
    |a_h(\xi)-\xi^2|\leq \sum _{k\in F}a(k,h)|\xi|^k, \, \xi\in \left [-\frac\pi h,\frac\pi h\right ].
\end{equation}

\item The two-grid scheme. The two-grid scheme can be adapted to the nonlinear frame as follows. Consider the equation
\begin{equation}\label{eq.two.grid}
    \left\{
\begin{array}{cc}
 \displaystyle i u^{0,h}_t + \Delta_h u^{0,h}=\bpi^{4h}_{h} f((\bpi^{4h}_{h})^* u^{0,h}), & t>0,  \\[8pt]
  \displaystyle u^{0,h}(0) =\bpi ^{4h}_h\varphi^h, &
\end{array}
    \right.
\end{equation}
where $(\bpi^{4h}_{h})^*:l^2(h\zz)\rightarrow l^2(4h\zz)$ is the adjoint of $\bpi^{4h}_{h}:l^2(4h\zz)\rightarrow l^2(h\zz)$ and $\varphi^h$ is an approximation of $\varphi$.

By \cite{liv2}, Theorem 4.1, {\it for any $p\in (0,4)$ there exists of a positive time $T_0=T_0(\|\varphi\|_{L^2(\rr)})$ and a unique solution $u^{h,0}\in C(0,T_0;\,l^2(h\zz^d))\cap
L^q(0,T_0;\,l^{p+2}(h\zz^d))$, $q=4(p+2)/p$, of the system \eqref{eq.two.grid}. Moreover, $u^{h,0}$ satisfies
\begin{equation}\label{mult140.0}
    \|u^h\|_{L^\infty(\rr,\,l^2(h\zz^d))}\leq \|\bpi ^{4h}_h \varphi
    ^h\|_{l^2(h\zz^d)}
\end{equation}
and
\begin{equation}\label{mult141.0}
    \|u^h\|_{L^{q}(0,T_0;\,l^{p+2}(h\zz^d))}\leq c(T_0) \|\bpi ^{4h}_h  \varphi
    ^h\|_{l^2(h\zz^d)},
\end{equation}
	where the above constant is independent of $h$.}

With $T_0$ obtained above, for any $k\geq 1$ we consider $u^{k,h}:[kT_0,(k+1)T_0]\rightarrow \cn$ the solution of the following system
\begin{equation}\label{eq.two.grid.2}
    \left\{
\begin{array}{cc}
 \displaystyle i u^{k,h}_t + \Delta_h u^{k,h}=\bpi^{4h}_{h} f((\bpi^{4h}_{h})^* u^{k,h}), & t\in [kT_0,(k+1)T_0],  \\[8pt]
  \displaystyle u^{k,h}(kT_0) =\bpi ^{4h}_h u^{k-1,h}(kT_0). &
\end{array}
    \right.
\end{equation}
Once, $u^{k,h}$ are computed the approximation $u^h$ of NSE is defined as
\begin{equation}\label{eq.two.grid.3}
u^h(t)=u^{k,h}(t), \, t\in [kT_0,(k+1)T_0).
\end{equation}

We point out that systems \eqref{eq.two.grid} and \eqref{eq.two.grid.2} have always a global  solution in the class $C(\rr,\,l^2(h\zz))$ (use the embedding 
$ l^2(h\zz)\subset l^\infty(h\zz)$, a classical fix point argument and the conservation of the $l^2(h\zz)$-norm). However, estimates in the $L^q(0,T;\,l^r(h\zz))$-norm, uniformly with respect to the mesh-size parameter $h >0$,  cannot be proved without using Strichartz estimates given by Theorem \ref{multstr}. Thus we need to take initial data obtained through  a two-grid process. Since the two-grid class of functions is not invariant under the  flow of  system \eqref{eq.two.grid} we need to update the solution at some  time-step $T_0$ which depends only on $L^2(\rr)$-norm of the initial data $\varphi$.
\end{itemize}

\medskip

The following theorems give us the existence and uniqueness of solutions for the above systems as well as  quantitative dispersive
estimates of solutions $u^h$, similar to those obtained in Lemma \ref{stab.1} for the continuous NSE, uniformly on the mesh-size parameter $h>0$.

\begin{teo}\label{schnon}
Let $p\in (0,4)$, $f(u)=|u|^pu$ and $A_h$ be such that ${\rm{Re}}( i A_h\varphi,\varphi)_h\leq 0$ and \eqref{linfty.dis-h} holds.
 Then for
every $\varphi ^h \in l^2(h\zz)$, there exists a unique global
solution $u^h\in C(\rr,\,l^2(h\zz))$
of  (\ref{eq.dis.non}) which  satisfies
\begin{equation}\label{sch140}
    \|u^h\|_{L^\infty(\rr,\,l^2(h\zz))}\leq \|\varphi ^h \|_{l^2(h\zz)}.
\end{equation}
Moreover, there exist $c(p)>0$ and $C(p,q)>0$ such that  for any finite interval $I$ with $|I|\leq
T_0=c(p)\|\varphi^h\|_{l^2(h\zz)}^{-4p/(4-p)}$ 
\begin{equation}\label{sch141}
    \|u^h\|_{L^{q}(I,\,l^{r}(h\zz))}\leq C(p,q)
\|\varphi^h\|_{l^2(h\zz)},
\end{equation}
where $(q,r)$ is a $1/2$-admissible pair and the above constant is independent of $h$.
\end{teo}

\begin{proof}
Condition ${\rm{Re}}( i A_h\varphi,\varphi)_h\leq 0$ implies the $l^2(h\zz)$ stability property \eqref{l2-stable-h}.
Then local existence is obtained by using Strichartz estimates given by Proposition \ref{strfil} applied to the
operator $\exp(itA_h)$ and a classical fix point argument in a suitable Banach space
(see \cite{liv3} and \cite{liv-siam} for more details). The global existence of solutions and
estimate \eqref{sch140} are guaranteed by the property
${\rm{Re}}(i A_h\varphi,\varphi)_h\leq 0$, and that ${\rm{Re}}(i f(u_h),u_h)_h=0$ and the energy identity:
\begin{equation}\label{discrete.energy}
\frac {d}{dt}\|u^h(t)\|_{l^2(h\zz)}^2=2{\rm{Re}} (iA_hu^h,u^h)_h+2{\rm{Re}}(if(u^h),u^h)_h\leq 0.
\end{equation}
Once the global existence is proved, estimate \eqref{sch141} is obtained  in a similar manner as Lemma \ref{stab.1} and we will omit its proof.
\end{proof}

\begin{teo}\label{multschnon}
Let  $p\in (0,4)$ and $q=4(p+2)/p$. Then for all $h>0$ and for
every $\varphi ^h \in l^2(4h\zz)$, there exists a unique global
solution $ u^h\in C(\rr,\,l^2(h\zz))\cap
L^q_{loc}(\rr,\,l^{p+2}(h\zz^d)) $ of  \eqref{eq.two.grid}-\eqref{eq.two.grid.3}
which satisfies 
\begin{equation}\label{mult140}
    \|u^h\|_{L^\infty(\rr,\,l^2(h\zz))}\leq \|\bpi ^{4h}_h \varphi
    ^h\|_{l^2(h\zz)}.
\end{equation}
Moreover, there exist $c(p)>0$ and $C(p,q)>0$  such that  for any finite interval $I$ with $|I|\leq
T_0=c(p)\|\varphi^h\|_{l^2(h\zz)}^{-4p/(4-p)}$
\begin{equation}\label{mult141}
    \|u^h\|_{L^{q}(I,\,l^{p+2}(h\zz))}\leq C(p,q) \|\bpi ^{4h}_h  \varphi
    ^h\|_{l^2(h\zz)},
\end{equation}
where $(q,r)$ is a $1/2$-admissible pair and the above constant is independent of $h$.
\end{teo}

\begin{proof}
The  existence in the interval $(0,T_0)$, $T_0=T_0(\|\varphi^h\|_{l^2(h\zz)})$
for system \eqref{eq.dis.non} is obtained by using the Strichartz estimates given by Theorem
\ref{multstr} 
  and a classical fix point argument in a suitable Banach space
(see \cite{liv3} and \cite{liv-siam} for more details).

For any $k\geq 1$ the same arguments guarantee the local existence for
systems \eqref{eq.two.grid.2}. To prove that each system has solutions on an interval of length $T_0$ we have to prove a priori that
 the $l^2(h\zz)$-norm of $u^h$ does not increase.
The particular approximation we have introduced of the nonlinear term in \eqref{eq.two.grid}-\eqref{eq.two.grid.2} gives us (after multiplying these
equations by $u^{k,h}$ and taking the $l^2(h\zz)$-norm) that for any $t\in [kT_0,(k+1)T_0]$
$$\|u^{k,h}(t)\|_{l^2(h\zz)}=\|u^{k,h}(kT_0)\|_{l^2(h\zz)}\leq \|u^{k-1,h}(kT_0)\|_{l^2(h\zz)}$$
and then
$$\|u^{k,h}(t)\|_{l^2(h\zz)}\leq \|u^{0,h}(0)\|_{l^2(h\zz)}=\|\bpi ^{4h}_h\varphi^h\|_{l^2(h\zz)}.$$

This proves \eqref{mult140} and the fact that for any $k\geq 1$ system \eqref{eq.two.grid.2} has a solution on the whole
interval $[kT_0,(k+1)T_0]$. Estimate \eqref{mult141} is obtained locally on each interval $[kT_0,(k+1)T_0]$
together with the local existence result.
\end{proof}

Let us consider  $u^h$  the solution of the semidiscrete problem
 \eqref{eq.dis.non} and $u$ of the continuous one \eqref{scha869}. In the following theorem we
evaluate the difference between $u^h$ and $\th u$.

\begin{teo}\label{main}Let $p\in (0,4)$, $s\in (0,1/2)$, $f(u)=|u|^pu$ and $A_h$ be
as in Theorem \ref{schnon} satisfying \eqref{err.symbol.1}.  For any $\varphi\in  H^s(\rr)$, we consider
 $u^h$ and $u\in L^\infty(\rr,\,H^s(\rr))\cap
L^{q_0}_{loc}(\rr,B^{s}_{p+2,2}(\rr))$, $q_0=4(p+2)/p$  solutions of problems \eqref{eq.dis.non} and \eqref{scha869}, respectively.
 Then  for any $T>0$ there exists a positive constant $C(T,\|\varphi\|_{L^2(\rr)})$
such that
\begin{align}
 \label{est.main}   \|u^h &-\th u\|_{L^{q_0}(0,T;\,l^{p+2}(h\zz))}+\|u^h-\th u\|_{L^\infty (0,T;\,l^2(h\zz))} \\
 \nonumber&\leq C(T,\|\varphi\|_{L^2(\rr)},p)
\Big[\eps(s,h)\|u\|_{L^\infty(0,T;\,H^s(\rr))}+ \Big(h^{s}+\eps(s,h)\Big)
\|u\|_{L^{q_0}(0,T;\,B^{s}_{p+2,2}(\rr))}^{p+1}\Big]
\end{align}
holds for all $h>0$.
\end{teo}

In the case of the two-grid method, the solution $u^h$ of system \eqref{eq.two.grid} approximates the solution $u$ of the
NSE \eqref{scha869} and the error committed is given by the following theorem.

\begin{teo}\label{main.two.grid}Let $p\in (0,4)$, $s\in (0,1/2)$, $f(u)=|u|^pu$.  For any $\varphi\in  H^s(\rr)$, we consider
 $u^h$ and $u\in L^\infty(\rr,\,H^s(\rr))\cap
L^{q_0}_{loc}(\rr,B^{s}_{p+2,2}(\rr))$, $q_0=4(p+2)/p$, solutions of problems \eqref{eq.two.grid}-\eqref{eq.two.grid.3}
and \eqref{scha869}, respectively.
 Then  for any $T>0$ there exists a positive constant $C(T,\|\varphi\|_{L^2(\rr)})$
such that
\begin{align}
 \label{est.main.grid}   \|u^h &-\th u\|_{L^{q_0}(0,T;\,l^{p+2}(h\zz))}+\|u^h-\th u\|_{L^\infty (0,T;\,l^2(h\zz))}\\
 \nonumber&\leq C(T,\|\varphi\|_{L^2(\rr)},p)
\Big[h^{s/2}\|u\|_{L^\infty(0,T;H^s(\rr))}+ \Big(h^{s}+h^{s/2}\Big)
\|u\|_{L^{q_0}(0,T;\,B^{s}_{p+2,2}(\rr))}^{p+1}\Big]
\end{align}
holds for all $h>0$.
\end{teo}

\begin{remark}
Using classical results on the solutions of the NSE (see for example \cite{MR559676}, Theorem~5.1.1, Ch.~5, p.~147) we can state the above result in a more compact way: For any $T>0$ there exists a positive constant $C(T,\|\varphi\|_{H^s(\rr)})$ such that
\begin{align}
 \label{est.main.grid.100}   \|u^h -\th u\|_{L^{q_0}(0,T;\,l^{p+2}(h\zz))}+\|u^h-\th u\|_{L^\infty (0,T;\,l^2(h\zz))}\leq C(T,\|\varphi\|_{H^s(\rr)}) h^{s/2}
 \end{align}
holds for all $h>0$.

\end{remark}

Theorem \ref{main} shows that if $h^s\leq \eps(s,h)$ then the error committed to
 approximate the nonlinear problem is
the same as for the linear problem with the same initial data.
As we proved in Section  \ref{examples}, for the higher order dissipative scheme $A_h=\Delta_h-i h^{2(m-1)}(-\Delta_h)^m$, $m\geq 2$,
and for the two-grid method, $\eps(s,h)=h^{s/2}\geq h^s$. So these schemes enter in this framework.
 It is also remarkable that the use of
dispersive schemes  allows to prove the convergence  for the NSE and to
obtain the convergence rate for $H^s(\rr)$ initial data with $0<s<1/2$.
 We point out that the
energy method does not provide any error estimate in this case, the minimal smoothing
required for the energy method being $H^{s}(\rr)$, with $s>1/2$ (see Section \ref{nondisipative} for all the details).

\medskip

In the following we prove Theorem \ref{main}, the proof of Theorem \ref{main.two.grid} being similar since the estimates
in any interval $(0,T)$ are obtained reiterating the argument in each interval $(kT_0,(k+1)T_0)$, $k\geq 0$, for some
 $T_0=T_0(\|\varphi\|_{L^2(\rr)})$ in view of the structure of the scheme.

\begin{proof}[Proof of Theorem \ref{main}]
The idea of the proof is that there exists a time $T_1$ depending on the $L^2(\rr)$-norm of the initial data:
 $$T_1\simeq \min\{1,\|\varphi\|_{L^2(\rr)}^{-4p/(4-p)}\},$$
  such that the
error in the approximation of the nonlinear problem
$$\err(t)=u^h(t)-\th u(t),$$  when considered in the ${L^{q_0}(0,T_1;\,l^{p+2}(h\zz))}\cap L^\infty(0,T_1;\,l^2(h\zz))$-norm
 is controlled by the error produced in the linear part
$$\errlin(t)=\exp(itA_h)\th\varphi-\th \exp(it\partial_x^2)\varphi.$$

In the following we denote by  $(q,r)$  one of the admissible pairs $(\infty,2)$ or $(q_0,p+2)$. We now
write the two solutions in the semigroup formulation given by systems \eqref{eq.dis.non} and \eqref{scha869}:
$$u^h(t)=\exp(itA_h)\th \varphi+i\int_0^t\exp(i(t-s)A_h)f(u^h(s))ds,$$
and, respectively,
$$\th u(t)=\th \exp(it\partial_x^2)\varphi+i\int _0^t \th \exp(i(t-s)\partial_x^2)f(u(s))ds.$$
Thus
\begin{equation}\label{error.1}
    \|\err\|_{L^q(0,T; \,l^r(h\zz))}\leq \|\errlin\|_{L^q(0,T; \,l^r(h\zz))}+\|\errnon\|_{L^q(0,T; \,l^r(h\zz))}
\end{equation}
where, by definition, 
$$\errnon(t)=\int_0^t \exp(i(t-s)A_h)f(u^h(s))ds-\int _0^t \th \exp(i(t-s)\partial_x^2)f(u(s))ds.$$

For the linear part  the  error  is estimated in Theorem
\ref{est.lin}:
\begin{equation}\label{estm.1}
\|\errlin\|_{L^q(0,T; \,l^r(h\zz))}\leq C(q)\eps(s,h)\max\{T,1\}\|\varphi\|_{H^s(\rr)}.
\end{equation}

 In the following we will estimate $\errnon$.
We write $\errnon(t)= I_2^h(t)+I_3^h(t)$ where
$$I_2^h(t)=\int_0^t \exp(i(t-s)A_h)\Big(f(u^h(s))-\th f(u(s))\Big)ds$$
and
$$I_3^h(t)=\int_0^t \Big(\exp(i(t-s)A_h)\th f(u(s))-\th \exp(i(t-s)\partial_x^2)f(u(s))\Big)ds.$$

\textbf{Step I. Estimate of $I_3^h$.}
For the last term, the inhomogeneous estimate \eqref{est.inh} in  Theorem \ref{est.lin}
and estimate \eqref{normebs} give us that
\begin{align}\label{est.10}
    \|I_3^h(t)\|_{L^q(0,T;\,l^r(h\zz))}&\leq C(q)\eps(s,h)\max\{1,T\} \|f(u)\|_{L^{q'_0}(0,T;\,B^{s}_{(p+2)',2}(\rr))}\\
  \nonumber  &\leq C(q)\eps(s,h)\max\{1,T\} T^{\frac {4-p(1-2s)}4}\|u\|_{L^q(0,T;B^s_{p+2,2}(\rr))}^{p+1}.
\end{align}

\textbf{Step II. Estimate of $I_2^h$.}
We now prove the existence of a time $T_0$ such that for all $T<T_0$,  $I_2^h$ satisfies
\begin{align}
\label{estm.2}\|I_2(t)&\|_{L^q(0,T;\,l^r(h\zz))}\\
\nonumber& \leq C(p) T^{1-\frac p4} \|\err\|_{L^{q_0}(0,T;\, l^{p+2}(h\zz))}\|\varphi\|_{L^2(\rr)}^p+
h^s T^{1-\frac p4} \|u\|_{L^{q_0}(0,T;B^s_{p+2,2}(\rr))}^{p+1}.
\end{align}

The inhomogeneous Strichartz's estimate \eqref{dis.stric.inhomg} applied to the operators $(\exp(itA_h))_{t\geq 0}$
shows that
\begin{align}\label{est.2000}
\|I_2^h(t)&\|_{L^q(0,T;\,l^r(h\zz))}\leq C(q) \|f(u^h)-\th f(u)\|_{L^{q_0'}(0,T;\,l^{(p+2)'}(h\zz))}\\
\nonumber&\leq C(q)\|f(u^h)- f(\th u)\|_{L^{q_0'}(0,T;\,l^{(p+2)'}(h\zz))}+C(q)\|f(\th u)-\th
f(u)\|_{L^{q_0'}(0,T;\,l^{(p+2)'}(h\zz))}.
\end{align}

We  evaluate each term in the right hand side of \eqref{est.200}.
In the case of the first one, applying H\"older's inequality in time we get
\begin{align*}
\|f(u^h)&- f(\th u)\|_{L^{q_0'}(0,T;\,l^{(p+2)'}(h\zz))}\\
&\leq T^{1-\frac p4}\|u^h-\th u\|_{L^{q_0}(0,T;\, l^{p+2}(h\zz))}\Big(\|u^h\|_{L^{q_0}(0,T;\, l^{p+2}(h\zz))}^p+\|\th u\|_{L^{q_0}(0,T;\, l^{p+2}(h\zz))}^p\Big).
\end{align*}

Let us now set $T_0$ as it is given by Lemma \ref{stab.1} and Theorem \ref{schnon}:
$$T_0\simeq \|\varphi\|_{L^2(\rr)}^{-\frac{4p}{4-p}}.$$
Thus, by Theorem \ref{exis.uniq}, Lemma \ref{stab.1} and Theorem \ref{multschnon} both  $u^h$ and $\th u$ have their $L^q(0,T;\,l^r(h\zz))$-norm  controlled by the $L^2$-norm of the initial data:
$$\|u^h\|_{L^{q_0}(0,T_0;\, l^{p+2}(h\zz))}\leq C(p)\|\varphi\|_{L^2(\rr)}$$
and
$$\|\th u\|_{L^{q_0}(0,T_0;\, l^{p+2}(h\zz))}\leq C(p)\|u\|_{L^{q_0}(0,T_0;\, L^{p+2}(\rr))}\leq C(p)\|\varphi\|_{L^2(\rr)}.$$
These estimates  show that for any $T<T_0$ the following holds:
\begin{align}\label{est.2001}
\|f(u^h)- f(\th u)& \|_{L^{q_0'}(0,T;\,l^{(p+2)'}(h\zz))}\leq C(p)T^{1-\frac p4}\|u^h-\th u\|_{L^{q_0}(0,T;\, l^{p+2}(h\zz))}
\|\varphi\|_{L^2(\rr)}^p.
\end{align}

It remains to estimate the second term in the right hand side of
\eqref{est.2000}.
We will use now the following result which will be proved in Section \ref{putos.lemas}.
\begin{lemma}\label{lemma.non}
Let $s\in [0,1]$, $p\geq 0$ and $f(u)=|u|^pu$. Then there exists a positive constant $c(p,s)$ such that
\begin{equation}\label{inter.1}
    \|f(\th u)-\th f(u)\|_{l^{(p+2)'}(h\zz)}\leq c(p,s)h^{s}\|u\|_{W^{s,p+2}(\rr)}^{p+1}
\end{equation}
holds for all $u\in W^{s,p+2}(\rr)$ and $h>0$.
\end{lemma}
Using this lemma, H\"older inequality in time and the embedding $B^s_{p+2,2}(\rr)\hookrightarrow W^{s,p+2}(\rr)$ (\cite{1055.35003}, Remark 1.4.3) we obtain:
\begin{align}\label{est.2003}\|f(\th u)-\th
f(u)\|_{L^{q_0'}(0,T;\,l^{(p+2)'}(h\zz))}&\leq c(p,s) h^{s}T^{1-\frac{p}4}\|u\|_{L^{q_0}(0,T;\,W^{s,p+2}(\rr))}^{p+1}\\
\nonumber& \leq c(p,s) h^{s}T^{1-\frac{p}4}\|u\|_{L^{q_0}(0,T;\,B^s_{p+2,2}(\rr))}^{p+1}.
\end{align}

Both \eqref{est.2001} and \eqref{est.2003} show  that $I_2(t)$ satisfies \eqref{estm.2}.

\medskip
\textbf{Step III. Estimate of $\err$.}
Collecting estimates \eqref{estm.1}, \eqref{est.10} and \eqref{estm.2} for both $(q,r)=(q_0,p+2)$ and
$(q,r)=(\infty,2)$ we obtain that for any $T<T_0$ the error $\err$ satisfies:
\begin{align}\label{error.2}
 \|\err&\|_{L^{q_0}(0,T;\,l^{p+2}(h\zz))}+\|\err\|_{L^\infty (0,T;\,l^2(h\zz))} \\
\nonumber \leq & C(p) \max\{1,T\}\eps(s,h)\|\varphi\|_{H^s(\rr)}+C(p)\|\err\|_{L^{q_0}(0,T;\, l^{p+2}(h\zz))} T^{1-p/4}\|\varphi\|_{L^2(\rr)}^p \\
\nonumber &+h^{s}T^{1-\frac p4}\|u\|_{L^{q_0}(0,T;\,B^{s}_{p+2,2}(\rr))}^{p+1}+\eps(s,h) \max\{1,T\}T^{\frac{4-p(1-2s)}4}
\|u\|_{L^{q_0}(0,T;B^s_{p+2,2}(\rr))}^{p+1}.
\end{align}
Now, let us set
$T_1\leq \min\{1,T_0\}$ such that
$$C(p)T_1^{1-p/4}\|\varphi\|_{L^2(\rr)}^p \leq \frac 12.$$
Then the error term $\err$ in the right hand side of \eqref{error.2} is absorbed in the left hand side:
\begin{align*}
\|\err  \|_{L^{q_0}(0,T_1;\,l^{p+2}(h\zz))}&+\|\err\|_{L^\infty (0,T_1;\,l^2(h\zz))}\\
&\leq C(p)\eps(s,h)\|\varphi\|_{H^s(\rr)}+C(p)
\|u\|_{L^{q_0}(0,T;B^s_{p+2,2}(\rr))}^{p+1}\Big(h^{s}+\eps(s,h) \Big).
\end{align*}

We now obtain the same estimate in any interval $(0,T)$.
Using that the $L^2(\rr)$-norm of the solution $u$ is conserved in time  we can apply
 the same argument in the interval $[kT_1,(k+1)T_1]$:
\begin{align*}
\|\err &\|_{L^{q_0}(kT_1,(k+1)T_1;\,l^{p+2}(h\zz))}+\|\err\|_{L^\infty (kT_1,(k+1)T_1;\,l^2(h\zz))}\\
&\leq C(p) \eps(s,h)\|u(kT_1)\|_{H^s(\rr)}+C(p)\Big(h^{s}+\eps(s,h)\Big)\|u\|_{L^{q_0}(kT_1,(k+1)T_1;\,B^{s}_{p+2,2}(\rr))}^{p+1}.
\end{align*}
Let us choose $T>0$ and $N\geq 1$ an integer such that $(N-1)T_1\leq T< NT_1 $. Thus
\begin{align*}
\|& \err \|_{L^{q_0}(0,T;\,l^{p+2}(h\zz))}+\|\err\|_{L^\infty (0,T;\,l^2(h\zz))}\\
&\leq \sum_{k=0}^{N-1}\Big(\| \err \|_{L^{q_0}(kT_1,(k+1)T_1;\,l^{p+2}(h\zz))}+\|\err\|_{L^\infty (kT_1,(k+1)T_1;\,l^2(h\zz))}\Big)\\
&\leq C(p)\eps(s,h)\sum_{k=0}^{N-1} \|u(kT_1)\|_{H^s(\rr)}+C(p)\Big(h^{s}+\eps(s,h)\Big)\sum_{k=0}^{N-1}
\|u\|_{L^{q_0}(kT_1,(k+1)T_1;\,B^{s}_{p+2,2}(\rr))}^{p+1}.
\end{align*}
Using that $(p+1)/q_0<1$ we have by the discrete H\"older's inequality that
$$\sum_{k=0}^{N-1}
\|u\|_{L^{q_0}(kT_1,(k+1)T_1;\,B^{s}_{p+2,2}(\rr))}^{p+1}\leq N^{1-\frac {p+1}{q_0}}\|u\|_{L^{q_0}(0,T;\,B^{s}_{p+2,2}(\rr))}^{p+1}.$$
Thus the error satisfies:
\begin{align*}
\| \err & \|_{L^{q_0}(0,T;\,l^{p+2}(h\zz))}+\|\err\|_{L^\infty (0,T;\,l^2(h\zz))}\\
&\leq  N \eps(s,h)\|u\|_{L^\infty(0,T;H^s(\rr))}+ \Big(h^{s}+\eps(s,h)\Big)
N^{1-\frac {p+1}{q_0}}\|u\|_{L^{q_0}(0,T;\,B^{s}_{p+2,2}(\rr))}^{p+1}\\
&\leq N \Big[\eps(s,h)\|u\|_{L^\infty(0,T;H^s(\rr))}+ \Big(h^{s}+\eps(s,h)\Big)
\|u\|_{L^{q_0}(0,T;\,B^{s}_{p+2,2}(\rr))}^{p+1}\Big]\\
&\leq C(T,\|\varphi\|_{L^2(\rr)})\Big[\eps(s,h)\|u\|_{L^\infty(0,T;H^s(\rr))}+ \Big(h^{s}+\eps(s,h)\Big)
\|u\|_{L^{q_0}(0,T;\,B^{s}_{p+2,2}(\rr))}^{p+1}\Big].
\end{align*}
This finishes the proof of Theorem \ref{main}.
\end{proof}
\section{Nondispersive methods}\label{nondisipative}
\setcounter{equation}{0}

In this section we will consider a numerical scheme for which the operator $A_h$ has no uniform (with respect to the mesh size $h$) dispersive properties
of Strichartz type. Accordingly we may not use $L^q_t L^r_x$ estimates for
the linear semigroup $\exp(itA_h)$
 and all the possible convergence 
estimates need to be based on the fact that the solution $u$ of the continuous problem is uniformly bounded in space and time:
$u\in L^\infty((0,T); L^\infty(\rr))$.  Thus, the only estimates we can use are those that the $L^2$-theory may yield.
When working with $H^s(\rr)$-data with $s>1/2$, using $L^\infty(\rr;H^s(\rr))$ estimates on solutions and
Sobolev's embedding we can get $L^2$-estimates.

There is a classical argument that works whenever  the  nonlinearity $f$ satisfies
\begin{equation}\label{lipshitz}
|f(u)-f(v)|\leq C (|u|^p+|v|^p)|u-v|.
\end{equation}
Standard error estimates (see Theorem \ref{est.lin} with the particular case $(q,r)=(\infty,2)$ or \cite{MR1005330}, Theorem 10.1.2, p. 201) and Gronwall's inequality yield when  $0\leq t\leq T$:
\begin{align}\label{star}
 \|u^h(t)-\th u(t)\|_{l^2(h\zz)}& \\
\nonumber\leq h^{1/2}C(T)\big(&\|\varphi\|_{H^1(\rr)}+ \|u\|_{L^\infty(0,T:\,H^1(\rr))}^{p+1}\big)\exp(T\|u\|_{L^\infty(0,T;\,H^1(\rr))}^p),
\end{align}
for the conservative semi-discrete finite-difference scheme. For the sake of completness we will prove this estimate in Section \ref{classical.smooth}.

We emphasize that in order to obtain estimate \eqref{star} we need to use that the solution $u$, which we want to approximate, belongs to the space $L^\infty(\rr)$, condition which is guaranteed by assuming that the initial data is smooth enough.  
However, obviously, in general, solutions of the NSE do not belong to $L^\infty(\rr)$ and therefore these estimates can not
be applied.
One can overcome this drawback assuming that the initial data belong to $H^1(\rr)$ or even to  $H^s(\rr)$ with $s>1/2$ since in this case  $H^s(\rr)\hookrightarrow L^\infty(\rr)$. Using
$H^1$-energy estimates and Sobolev's embedding we can deduce $L^\infty$-bounds on solutions allowing to apply \eqref{star}.
We emphasize  that this standard approach fails to provide any error estimate for initial data in $H^s(\rr)$
with $s<1/2$.

However,  this type of error estimate can also be used for $H^s(\rr)$-initial data
with $s<1/2$ (or even for $L^2(\rr)$-initial data), by a density argument.
Indeed, given $\varphi\in H^s(\rr)$ with $0\leq s<1/2$, for any $\delta>0$ we may choose $\varphi_\delta\in H^1(\rr)$
such that
$$\|\varphi-\varphi_\delta\|_{H^s(\rr)}\leq \delta.$$
Let $u_\delta$ be the solution of NSE corresponding to $\varphi_\delta$. Obviously, $\varphi_\delta$ being $H^1(\rr)$-smooth, we can apply standard results as \eqref{star} to $u_\delta$. On the other hand, stability results for NSE allow us to prove the
proximity of $u$ and $u_\delta$ in $H^s(\rr)$. This allows showing the convergence of numerical approximations of $u_\delta$, that we may denote by $u_{\delta,h}$,
towards the solution $u$ associated to $\varphi$ as both $\delta\rightarrow 0$ and $h\rightarrow 0$. But for this to be true $h$
needs to be exponentially small of the order of $\exp(-1/\delta)$ which is much smaller than the typical mesh-size needed to apply
the results of the previous sections on dispersive schemes that required $h$ to be of the order of $\delta^{2/s}$.

\subsection{A classical argument for  smooth initial data}\label{classical.smooth}

In this section we  present the technical details of the error estimates in the case of $H^1(\rr)$-initial data. In this
case we do not require the numerical scheme to be dispersive, the only ingredient being the Sobolev's embedding
$H^1(\rr)\hookrightarrow L^\infty(\rr)$.

\begin{teo}\label{regularizare}
Let  $f(u)=|u|^pu$ with $p\in (0,4)$ and $u\in C(\rr,\,H^1(\rr))$ be solution of \eqref{scha869} with initial data
$\varphi\in H^1(\rr)$. Also assume that $A_h$ is  an
approximation of order two of the laplace operator $\partial_x^2$ and $u^h$ is the solution of the following system
\begin{equation}\label{eq.dis.non.1}
    \left\{
\begin{array}{cc}
 \displaystyle i u^h_t + A_h u^h=f(u^h), & t>0,  \\[8pt]
  \displaystyle u^h(0) =\th \varphi, &
\end{array}
    \right.
\end{equation}
satisfying $\|u^h\|_{L^\infty((0,T)\times h\zz)}\leq C(T,\|\varphi\|_{H^1(\rr)})$.

Then for all $T>0$ and $h>0$
\begin{equation}
\|u^h(t)-\th u(t)\|_{l^2(h\zz)}\leq h^{1/2}\max\{T,T^2\}\big(\|\varphi\|_{H^1(\rr)}+ \|u\|_{L^\infty(0,T:\,H^1(\rr))}^{p+1}\big)\exp(T\|u\|_{L^\infty(0,T;\,H^1(\rr))}^p).
\end{equation}
\end{teo}

We now give an example where the hypotheses of the above theorem are verified.
We consider the following NSE:
\begin{equation}\label{scha869.1}
\left\{
\begin{array}{l}
  iu_t+\partial_x^2 u = |u|^pu,  \,x\in \rr,\,t> 0,\\
  u(0,x)  =\varphi(x),\, x\in \rr,
  \end{array}\right.
\end{equation}
and its numerical approximation
\begin{equation}\label{eq.dis.non.1111}
    \left\{
\begin{array}{cc}
 \displaystyle i u^h_t + \Delta_h u^h=|u^h|^pu^h, & t>0,  \\[8pt]
  \displaystyle u^h(0) =\varphi^h. &
\end{array}
    \right.
\end{equation}
In the case of the continuous problem we have the following conservation laws (see \cite{1055.35003}, Corollary~4.3.4, p.~93):
$$\|u(t)\|_{L^2(\rr)}=\|\varphi\|_{L^2(\rr)}$$
and
$$\frac{d}{dt}\Big(\frac 12\int _{\rr}| u_x(t,x)|^2dx +\frac 1{p+2}\int _{\rr}|u(t,x)|^{p+2}dx\Big) =0.$$
 The same identities apply in the semi-discrete case (it suffices to multiply the equation \eqref{eq.dis.non.1111} by $\overline{u}^h$, respectively $\overline{u}^h_t$, to sum over the integers and to take the real part of the resulting identity):
 $$\|u^h(t)\|_{l^2(h\zz)}=\|\varphi^h\|_{l^2(h\zz)}$$
and
$$\frac d{dt}\Big( \frac h2 \sum_{j\in \zz} \big|\frac{u^h_{j+1}(t)-u^h_j(t)}h\big|^2+\frac h{p+2}\sum_{j\in \zz}|u^h_{j}(t)|^{p+2}  \Big)=0.$$
In view of the above identities, the hypotheses of Theorem \ref{regularizare} are verified.

\begin{proof}[Proof of Theorem \ref{regularizare}]
Using the variations of constants formula we get
$$\th u(t)=\th \exp(it\partial_x^2)\varphi+\int _0^t \th \exp(i(t-\sg)\partial_x^2)f(u(\sg))d\sg$$
and
$$u^h(t)=\exp(itA_h)\th \varphi +\int _0^t \exp(i(t-\sg)A_h)f(u^h(\sg))d\sg.$$
Then
\begin{align}\label{error.non.grow}
\err(t)&:=\|u^h(t)-\th u(t)\|_{l^2(h\zz)}\\
\nonumber&\leq \|\exp(itA_h)\th \varphi-\th \exp(it\partial_x^2)\varphi\|_{l^2(h\zz)}\\
\nonumber&\hspace{1cm}+\int_0^t \|\exp(i(t-\sg)A_h)\Big(f(u^h(\sg))-\th f(u(\sg))\Big)d\sg\|_{l^2(h\zz)}d\sg\\
\nonumber&\hspace{1cm}+\int _0^t \|\exp(i(t-\sg)A_h)\th f(u(\sg))-\th \exp((t-\sg)\partial_x^2)f(u(\sg))\|_{l^2(h\zz)}d\sg.
\end{align}
Now, applying the error estimates for the linear terms as in \eqref{est.linf.l2} with $\eps(1,h)=h^{1/2}$, we get
\begin{equation}\label{est.500}
\|\exp(itA_h))\th \varphi-\th \exp(it\partial_x^2)\varphi\|_{l^2(h\zz)}\leq T h^{1/2}\|\varphi\|_{H^1(\rr)}.
\end{equation}
Also, using that $f(u)=|u|^pu$ we have that $\|f(u)\|_{H^1(\rr)}\leq C \|u\|_{H^1(\rr)}^p$ and then
by \eqref{est.linf.l2} we  get
\begin{align}\label{est.501}
\int _0^t \|\exp(i(t-\sg)A_h)\th f(u(\sg))-&\th \exp(i(t-\sg)\partial_x^2)f(u(\sg))\|_{l^2(h\zz)}d\sg\\
\nonumber&\leq C T h^{1/2} \|f(u)\|_{L^1(0,T;\,H^1(\rr))}\\
\nonumber&\leq  C T^2 h^{1/2} \|u\|_{L^\infty(0,T;\,H^1(\rr))}^{p+1}.
\end{align}
Using the $l^2(h\zz)$-stability of $\exp(itA_h)$, \eqref{error.non.grow}, \eqref{est.500} and \eqref{est.501} we obtain
$$\err(t)\leq T h^{1/2}\|\varphi\|_{H^1(\rr)} +CT^2 h^{1/2} \|u\|_{L^\infty(0,T;\,H^1(\rr))}^{p+1}+
\int _0^t \|f(u^h(\sg))-\th f(u(\sg)\|_{l^2(h\zz))}.$$
Now we write $f(u^h(s))-\th f(u(s))=I_1^h(s)+I_2^h(s)$
where
$$I_1^h(s)=f(u^h(s))-f(\th u(s) ),\,I_2^h(s)=f(\th u(s))-\th f(u(s)).$$
In the case of $I_1^h$ we use that $f$ satisfies \eqref{lipshitz} 
to get
\begin{align*}
\|I_1^h(s)\|_{l^2(h\zz)}
&\leq C\Big(\|u^h(s)\|_{l^\infty(h\zz)}^p+\|\th u(s)\|_{l^\infty(h\zz)}^p\Big)\|u^h(s)-\th u(s)\|_{l^2(h\zz)}\\
&\leq C(\|u^h\|_{L^\infty((0,T)\times h\zz)}^p+\|u\|_{L^\infty((0,T)\times \rr)}^p)
\|u^h(s)-\th u(s)\|_{l^2(h\zz)}\\
&\leq C\|u\|_{L^\infty(0,T;\, H^1(\rr))}^p\err(s).
\end{align*}
Using the same arguments as in Lemma \ref{lemma.non} we obtain that
$$\|I_2^h(s)\|_{l^2(h\zz)}\leq h\|u(s)\|_{H^1(\rr)}^{p+1}.$$

Putting together all the above estimates, for any $0\leq t\leq  T$ we obtain:
\begin{align*}
\err(t)&\leq h^{1/2}T\|\varphi\|_{H^1(\rr)}+\|u\|_{L^\infty(0,T;\, H^1(\rr))}^p
\int_0^t \err(\sg)d\sg\\
&\hspace{1cm}+hT\|u\|_{L^\infty(0,T;H^1(\rr))}^{p+1}+T^2 h^{1/2} \|u\|_{L^\infty(0,T:\,H^1(\rr))}^{p+1}\\
&\leq h^{1/2}\max\{T,T^2\}\big(\|\varphi\|_{H^1(\rr)}+ \|u\|_{L^\infty(0,T:\,H^1(\rr))}^{p+1}\big)+\|u\|_{L^\infty(0,T;\,H^1(\rr))}^p
\int_0^t \err(s)ds.
\end{align*}
Applying Gronwall's Lemma we obtain
\begin{align}
\label{error.smoot}\err(t)
\lesssim h^{1/2}\max\{T,T^2\}\big(\|\varphi\|_{H^1(\rr)}+ \|u\|_{L^\infty(0,T:\,H^1(\rr))}^{p+1}\big)\exp(T\|u\|_{L^\infty(0,T;\,H^1(\rr))}^p).
\end{align}
The proof is now finished.
\end{proof}

\subsection{Approximating $H^s(\rr)$, $s<1/2$, solutions by smooth ones.}
Given $\varphi\in H^s(\rr)$ we choose an approximation $\tvarphi\in H^1(\rr)$
 such that $\|\varphi-\tvarphi\|_{H^s(\rr)}$ is small (a similar analysis can be done by considering $\varphi_\delta\in H^{s_1}$ with
 $s_1>1/2$).
 For $\tvarphi$ we consider the following approximation of $\tilde u$
solution of the NSE \eqref{scha869} with initial data $\tvarphi$:
\begin{equation}\label{as}
    \left\{
\begin{array}{cc}
 \displaystyle i \partial_t \tuh(t) + A_h \tuh=f(\tuh), & t>0,  \\
 \\
  \displaystyle \tuh(0) =\th \tvarphi, &
\end{array}
    \right.
\end{equation}
where the operator $A_h$ is a second order  approximation of the Laplace operator. We do not require the linear scheme associated
to the operator $A_h$ to satisfy uniform dispersive estimates.

Solving \eqref{as} we  obtain an approximation $\tuh$ of the solutions $\tilde u$ of NSE with initial datum $\tvarphi$,
which itself is an approximation of the solution $u$ of the NSE with initial datum $\varphi$.

In the following Theorem we give an explicit estimate of the distance between $\tuh$ and $u$.
\begin{teo}\label{aprox.regular}
Let $0\leq s< 1/2$, $\varphi\in H^s(\rr)$, and $u\in C(\rr;\,H^s(\rr))$ be the solution of NSE with initial datum
$\varphi$ given by Theorem \ref{exis.uniq}.
For any $T>0$ there exists a positive constant $C(T,\|\varphi\|_{L^2(\rr)})$ such that the following holds
\begin{equation}\label{est.120}
   \|\th u-\tuh\|_{L^\infty(0,T;\,l^2(h\zz))}\leq C(T,p,\|\varphi\|_{L^2(\rr)})\|\varphi-\tvarphi\|_{L^2(\rr)} +h^{1/2}\exp\big(T\|\tilde u\|_{L^\infty(0,T;\,H^1(\rr))}^p\big)
\end{equation}
 for all $h>0$ and $\delta>0$.
\end{teo}

%
%


In the following we show that the above method of regularizing the initial data $\varphi\in H^s(\rr)$ and then applying the $H^1(\rr)$ theory for that
approximation does not give the same rate of convergence $h^{s/2}$ obtained in the case of a dispersive method of order two
(see \eqref{est.main.grid.100}). This occurs since for $\|\varphi-\tvarphi\|_{L^2(\rr)}$ to be small,  $\|\tvarphi\|_{H^1(\rr)}$ needs to be large and 
$\| \tilde u \|_{L^\infty(0,T;\,H^1(\rr))}$ too.

To simplify the presentation we will consider the case $p=2$.

\begin{teo}\label{aproximare.log}
Let $p=2$, $0<s<1/2$, $\varphi\in H^s(\rr)$ and $u\in C(\rr,H^s(\rr))$ be solution of NSE with initial data $\varphi$ given by
Theorem \ref{exis.uniq} and $u _h^*$ be the best approximation with $H^1(\rr)$-initial data as given by \eqref{as} with the conservative
approximation $A_h=\Delta_h$. Then for
any time $T$, there exists a constant $C(\|\varphi\|_{H^s(\rr)},T,s)$ such that
\begin{equation}\label{estimare.proasta}
    \|\th u-u_h^*\|_{L^\infty(0,T;\,l^2(h\zz))}\leq C(\|\varphi\|_{H^s(\rr)},T,s)|\log h|^{-\frac s {1-s}}.
\end{equation}

\end{teo}

To prove this result we will use in an essential manner the following Lemma.
\medskip
\begin{lemma}\label{minimizar}
Let $0<s<1$ and $h\in (0,1)$. Then for any $\varphi\in H^s(\rr)$ the functional $J_{h,\varphi}$ defined by
\begin{equation}\label{functional}
J_{h,\varphi}(g)=\frac 12 \|\varphi-g\|_{L^2(\rr)}^2+\frac h2\exp(\|g\|_{H^1(\rr)}^2)
\end{equation}
satisfies:
 \begin{equation}\label{estm.minim}
    \min _{g\in H^1(\rr)}J_{h,\varphi}(g)\leq C(\|\varphi\|_{H^s(\rr)},s)|\log h|^{- s/ (1-s)}.
 \end{equation}
 Moreover, the above estimate is optimal in the sense that the power of the $|\log h|$ term cannot be improved:
 for any $0<\epsilon<1-s$ there exists $\varphi_\epsilon \in H^s(\rr)$ such that
 \begin{equation}\label{estimate.below}
 \liminf_{h\rightarrow 0} \frac{ \min _{g\in H^1(\rr)}J_{h,\varphi_\eps}(g)}{|\log h|^{- (s+\eps)/ (1-s-\eps)}}>0.
 \end{equation}

\end{lemma}


\begin{remark}
We point out that, to obtain \eqref{estm.minim} and \eqref{estimate.below}, we will use in an essential manner that $s<1$. In fact in the case $s=1$ the
minimum of $J_h$ over $H^1(\rr)$ is of order $h$. This can be seen  by choosing $g=\varphi$ and observing that
 $J_h(\varphi)=h\exp(\|\varphi\|_{H^1(\rr)})$. This choice cannot be done if $\varphi\in H^s(\rr)\backslash H^1(\rr)$.
\end{remark}

\begin{proof}[Proof of Theorem \ref{aproximare.log}]
Let us choose $\tvarphi\in H^1(\rr)$ which approximates $\varphi$ in $H^s(\rr)$. Then
by Theorem \ref{aprox.regular} we get
\begin{align}\label{est.1212}
   \|\th u- \tuh\|_{L^\infty(0,T;\,l^2(h\zz))}^2& \leq C(T,\|\varphi\|_{L^2(\rr)}) \|\varphi-\tvarphi\|_{L^2(\rr)}^2 +h\exp\big(2T\|\tvarphi\|_{H^1(\rr)}^2\big)\\
\nonumber   &\leq C(T,\|\varphi\|_{L^2(\rr)})J_{h,\sqrt{2T}\varphi }(\sqrt{2T}\tvarphi),
\end{align}
where $\tuh$ is the solution of \eqref{as} with initial data $\th  \tvarphi$.

For each $h$ fixed,  in order to obtain the best approximation $ u_h^*$ of $\th u$, we have to choose in the right hand side of the above
inequality the function $\varphi^*$ which minimizes the  functional
$J_{h,\sqrt{2T}\varphi}(\cdot)$ defined by \eqref{functional}
over  $H^1(\rr)$. Using estimate \eqref{estm.minim} from Lemma \ref{minimizar} we obtain the desired result:
\begin{align*}
 \|\th u-u_h^*\|_{L^\infty(0,T;\,l^2(h\zz))}&\leq C(\|\varphi\|_{H^s(\rr)},T,s)\min _{\tvarphi\in H^1(\rr)}
J_{h,\sqrt{2T}\varphi }(\sqrt{2T}\tvarphi)\\
 &\leq C(\|\varphi\|_{H^s(\rr)},T,s)|\log h|^{-\frac s {1-s}}
 \end{align*}
where $u_h^*$ is the solution of  \eqref{as} with initial data $\th\varphi^*$
\end{proof}

\begin{proof}[Proof of Lemma \ref{minimizar}]
The  functional $J_{h,\varphi}$ is convex and its minimizer, $g_h$, is unique.
The function $g_h$ satisfies the following equation:
\begin{equation}\label{eq.minimizer}
-\varphi+g_h+h\exp(\|g_h\|_{H^1(\rr)}^2)(-\Delta g_h+g_h)=0
\end{equation}
and so
$$
    \left[I+h\exp(\|g_h\|_{H^1(\rr)}^2)(I-\Delta) \right]g_h=\varphi.
$$
Thus $c_h=\|g_h\|_{H^1(\rr)}$ is the unique  solution of
\begin{equation}\label{eq.ch}
    c_h=\Big\| (I-\Delta)^{1/2} \left[I+h\exp(c_h^2)(I-\Delta)\right]^{-1} \varphi\Big\|_{L^2(\rr)}.
\end{equation}

\textbf{Step I. A useful auxiliary function.} Let us consider the function $q_h(x)=hx^\beta \exp(x)-c$ for some positive constants $\beta$ and
$c$. We prove that there exist two constants $a_1(c)$ and $a_2(c)$ such that the solution $x_h$ of the equation $q_h(x)=0$ satisfies
\begin{equation}\label{cotas.xh}
    |\log h|-\beta \log |\log h|+a_1(c)\leq x_h \leq  |\log h|-\beta \log |\log h|+a_2(c).
\end{equation}
Let us choose a real number $a$. Using that $h=\exp(-|\log h|)$ we get:
\begin{align*}
q_h(( |\log h|-\beta \log |\log h|+a))&=\left( |\log h|-\beta \log |\log h|+a\right)^\beta \exp(-\beta\log|\log h| +a)-c\\
&=\left(1-\beta \frac {\log |\log h|}{|\log h|}+\frac a{|\log h|}\right)^\beta \exp(a)-c\\
&\rightarrow _{h\rightarrow 0} \exp(a)-c.
\end{align*}
Choosing now two constants $a_1$ and $a_2$ such that $\exp(a_1)<c<\exp(a_2)$ and using that the function $q_h$ is increasing
we obtain that, for  $h$ small enough, $x_h$, solution of $q_h(x)=0$,
satisfies \eqref{cotas.xh}.

\medskip

\textbf{Step II. Upper bounds on  $c_h$.}
Using that $\varphi\in H^s(\rr)$, identity \eqref{eq.ch}  gives us
\begin{align*}
    c_h&=\Big\| (I-\Delta)^{1/2} \left[I+h\exp(c_h^2)(I-\Delta)\right]^{-1} \varphi\Big\|_{L^2(\rr)}\\
    &=\Big\| (I-\Delta)^{(1-s)/2} \left[I+h\exp(c_h^2)(I-\Delta)\right]^{-1} (I-\Delta)^{s/2}\varphi\Big\|_{L^2(\rr)}\\
    &=(h e^{c_h^2})^{(s-1)/2}\Big\| \left[h e^{c_h^2}(I-\Delta)\right]^{(1-s)/2} \left[I+h e^{c_h^2}(I-\Delta)\right]^{-1}
     (I-\Delta)^{s/2}\varphi\Big\|_{L^2(\rr)}\\
     &\leq (h\exp(c_h^2))^{(s-1)/2}\|\varphi\|_{H^s(\rr)},
\end{align*}
since, when $s\in [0,1]$, the symbol in the Fourier variable of the operator 
$$ \left[h e^{c_h^2}(I-\Delta)\right]^{(1-s)/2} \left[I+h e^{c_h^2}(I-\Delta)\right]^{-1}$$
is less than one.

Then
$c_h^2(h \exp(c_h^2))^{1-s} \leq \|\varphi\|_{H^s(\rr)}^2$ and
\begin{equation}\label{eq.ch.1}
    (c_h^2)^{1/(1-s)}he^{c_h^2}\leq  \|\varphi\|_{H^s(\rr)}^{2/(1-s)}.
\end{equation}
Applying the result of Step I to $\beta =1/(1-s)$ and $c=\|\varphi\|_{H^s(\rr)}^{2/(1-s)}$ we obtain that $c_h$ satisfies:
\begin{equation}\label{main.est.ch}
c_h^2\leq |\log h|-\frac 1{1-s}\log |\log h|+a_2,
\end{equation}
for some constant $a_2=a_2(\|\varphi\|_{H^s(\rr)}^{2/(1-s)})$.
In particular, when $s<1$, 
$$h\exp(c_h^2)=\exp(c_h^2-|\log h|)\leq
\exp\Big(-\frac 1{1-s}\log |\log h|+a_2\Big)\rightarrow 0,$$ as $h\rightarrow 0$.

\medskip
\textbf{Step III. Estimates on $J_h(g_h)$.} Using that the minimizer $g_h$ satisfies  equation \eqref{eq.minimizer} and $c_h=\|g_h\|_{H^1(\rr)}$, we get
\begin{align*}
    2\min _{g\in H^1(\rr)}J_h(g)&=2J_h(g_h)=
    \|\varphi-g_h\|_{L^2(\rr)}+ h \exp(\|g_h\|_{H^1(\rr)}^2)\\
  &=  (h\exp(c_h^2))^2\|(I-\Delta)g_h\|_{L^2(\rr)}^2+h\exp(c_h^2)\\
    &=(h\exp(c_h^2))^2\|(I-\Delta)\left[I+h\exp(c_h^2)(I-\Delta)\right]^{-1}\varphi\|_{L^2(\rr)}^2+h\exp(c_h^2)\\
    &= (he^{c_h^2})^{s}
    \Big\|\left[he^{c_h^2}(I-\Delta)\right]^{1-s/2}\left[I+he^{c_h^2}(I-\Delta)\right]^{-1}(I-\Delta)^{s/2}\varphi\Big\|_{L^2(\rr)}^2+he^{c_h^2}\\
    &\leq (he^{c_h^2})^{s}\|\varphi\|_{H^s(\rr)}^2+he^{c_h^2}\leq (he^{c_h^2})^s\left(\|\varphi\|_{H^s(\rr)}^2+(he^{c_h^2})^{1-s}\right)\\
&\leq c(s,\|\varphi\|_{H^s(\rr)})(h\exp(c_h^2))^s,
\end{align*}
where in the last inequality we used that $s\leq 1$ and $h\exp(c_h^2)\rightarrow 0$ as $h\rightarrow 0$.

Thus, by \eqref{main.est.ch} we obtain that
\begin{equation}\label{min.j}
    \min _{g\in H^1(\rr)}J_h(g)\leq c(s,\|\varphi\|_{H^s(\rr)}) (h\exp(c_h^2))^s\leq c(s,\|\varphi\|_{H^s(\rr)}) |\log h|^{-\frac s{1-s}}.
\end{equation}

\medskip
\textbf{Step IV. A particular function $\varphi$.} Let us choose $\eps>0$  and $\varphi_\eps$ be defined by means of its Fourier transform
$$\hat \varphi_\eps^2(\xi)=\frac 1{(1+\xi^2)^{s+\frac 12+\eps}}.$$
Thus, for any $\eps>0$,  $\varphi_\eps\in H^s(\rr)$. We will prove that, in this case, the solution  $c_{\eps,h}$ of \eqref{eq.ch} satisfies
\begin{equation}\label{est.c.eps.h}
    c_{\eps,h}^2\geq |\log h|-\frac 1{1-s-\eps}\log |\log h|+a_1
\end{equation}
and
\begin{equation}\label{est.min.eps.h}
    \min _{g\in H^1(\rr)}J_{h,\varphi_\eps}(g)\geq  (h\exp(c_{\eps,h}^2))^{s+\eps} \geq |\log h|^{- (s+\eps)/ (1-s-\eps)},
\end{equation}
for some constant $a_1$.

To prove \eqref{est.c.eps.h} and \eqref{est.min.eps.h} we
claim that for any $\gamma \in (-1/2,2)$ and $x$ large enough the following holds:
\begin{equation}\label{ineg.tonta}
    \int _{\rr} \frac{(1+\xi^2)^\gamma}{(x+1+\xi ^2)^2}d\xi \geq \frac {c(\gamma)}{x^{3/2-\gamma}}.
\end{equation}

Using that $c_{\eps,h}$ is solution of \eqref{eq.ch} and estimate \eqref{ineg.tonta} with $\gamma=1/2-s-\eps$ and $x=(h\exp{c_{\eps,h}^2})^{-1}$
we obtain
\begin{align*}
c_{\eps,h}^2&=\int _\rr \frac {(1+\xi^2)\hat \varphi_\eps^2(\xi)}{(1+h\exp(c_{\eps,h}^2)(1+\xi^2))^2}d\xi\\
&=\frac 1{(h\exp(c_{\eps,h}^2))^2}\int _\rr \frac {(1+\xi^2)^{\frac 12-s-\eps}}{\Big((h\exp(c_{\eps,h}^2))^{-1}+(1+\xi^2)\Big)^2}d\xi\\
&\geq \frac 1{(h\exp(c_{\eps,h}^2))^{1-s-\eps}}
\end{align*}
and 
$$h\exp(c_{\eps,h}^2) (c_{\eps,h}^2 )^{1/(1-s-\eps)}-1\geq 0.$$

Applying Step I to the function $q_h=hx^{1/(1-s-\eps)}\exp(x)-1$  we find that
\begin{equation}\label{estm.ch.eps}
    c_{\eps,h}^2\geq |\log h|-\frac 1{1-s-\eps}\log |\log h|+a_1,
\end{equation}
for some constant $a_1$.

This concludes the proof of \eqref{est.c.eps.h}.

We now prove \eqref{est.min.eps.h}. In view of  \eqref{eq.minimizer}  the minimizer $g_{\eps,h}$ 
satisfies
\begin{equation}\label{gepsh}
-\varphi_\eps+g_{\eps,h}+h\exp(\|g_{\eps,h}\|^2_{H^1(\rr)})(-\Delta g_{\eps,h}+g_{\eps,h})=0
\end{equation}
and
\begin{equation}\label{gepsh.explicit}
g_{\eps,h}= \left[I+h \exp(c_{\eps,h}^2)(I-\Delta)\right]^{-1}\varphi_\eps,
\end{equation}
where $c_{\eps,h}=\|g_{\eps,h}\|_{H^1(\rr)}.$

Thus 
\begin{align*}
2J_{h,\varphi_\eps}(g_{\eps,h})&=
\|\varphi_{\eps}-g_{\eps,h}\|_{L^2(\rr)}^2+h\exp(\|g_{\eps,h}\|^2_{H^1(\rr)})\\
&=(h\exp(c_{\eps,h}^2))^2\|(I-\Delta )g_{\eps,h}\|_{L^2(\rr)}^2+h\exp(c_{\eps,h}^2)\\
&=(h\exp(c_{\eps,h}^2))^2\|(I-\Delta ) \left[I+h \exp(c_{\eps,h}^2)(I-\Delta)\right]^{-1}\varphi_\eps\|_{L^2(\rr)}^2+h\exp(c_{\eps,h}^2).
\end{align*}
Writing the last term in Fourier  variable we get
\begin{align*}
2\min _{g\in H^1(\rr)}J_{h,\varphi_\eps}(g)&=(h\exp(c_{\eps,h}^2))^2\int_{\rr}\frac {(1+\xi^2)\hat \varphi_{\eps,h}^2(\xi)}
{\Big(1+h\exp(c_{\eps,h}^2)(1+\xi^2)\Big)^2}+h\exp(c_{\eps,h}^2)\\
&=\int _\rr\frac {(1+\xi^2)^{\frac 32-s-\eps}}{\Big((h\exp(c_{\eps,h}^2))^{-1}+1+\xi^2)\Big)^2} d\xi+h\exp(c_{\eps,h}^2).
\end{align*}
The same arguments as in Step II give us that $h\exp(c_{\eps,h}^2)\rightarrow 0$ as 
$h\rightarrow 0$. Then for small enough $h$, $x_h$ defined by $x_h=(h\exp(c_{\eps,h}^2))^{-1}$
is sufficiently large to apply inequality
\eqref{ineg.tonta} with $\gamma=3/2-s-\eps$. We get
\begin{align*}
2\min _{g\in H^1(\rr)}J_{h,\varphi_\eps}(g)\geq (h\exp(c_{\eps,h}^2))^{s+\eps}+h\exp(c_{\eps,h}^2)\geq (h\exp(c_{\eps,h}^2))^{s+\eps}.
\end{align*}
Using now \eqref{estm.ch.eps} we obtain
$$\min _{g\in H^1(\rr)}J_{h,\varphi_\eps}(g)\gtrsim |\log h|^{-\frac{s+\eps}{1-s-\eps}}$$ which proves \eqref{estimate.below}.

To finish the proof it remains to prove \eqref{ineg.tonta}. For $|x|\rightarrow \infty$, using changes of variables we get:
\begin{align*}
\int _{0}^\infty \frac{(1+\xi^2)^\gamma}{(x+1+\xi ^2)^2}d\xi& \gtrsim \int_{\sqrt 3}^\infty \frac{(1+\xi^2)^\gamma}{x^2+(1+\xi^2)^2}d\xi+O(\frac 1{x^2})\\
&\mathop{=}_{\xi=(\mu^2-1)^{1/2}}\int_2^\infty \frac {\mu^{2\gamma}}{x^2+\mu^4}\frac {\mu }{(\mu^2-1)^{1/2}}d\mu +O(\frac 1{x^2})\\
&\gtrsim \int_2^\infty \frac {\mu^{2\gamma}}{x^2+\mu^4}d\mu +O(\frac 1{x^2})\\
&\mathop{=}_{\mu=x^{1/2}\xi}\frac 1{x^{3/2-\gamma}}\int _{\xi \geq x^{-1/2}} \frac {\xi^{2\gamma}}{1+\xi^4}d\xi+O(\frac 1{x^2})\\
&\gtrsim \frac 1{x^{3/2-\gamma}}+O(\frac 1{x^2})\gtrsim \frac 1{x^{3/2-\gamma}}(1+\frac 1{x^{1/2+\gamma}})\\
&\gtrsim \frac 1{x^{3/2-\gamma}}
\end{align*}
which proves \eqref{ineg.tonta}.
\end{proof}

\begin{proof}[Proof of Theorem \ref{aprox.regular}]
Using the stability result \eqref{est.nse.1} for the  NSE we obtain
\begin{align*}\|u-\tu\|_{L^\infty(0,T;\,L^2(\rr))}&\leq C(T,p,\|\varphi\|_{L^2(\rr)},\|\tvarphi\|_{L^2(\rr)})\|\varphi-\tvarphi\|_{L^2(\rr)}\\
&\leq C(T,p,\|\varphi\|_{L^2(\rr)}) \|\varphi-\tvarphi\|_{L^2(\rr)}.
\end{align*}

Now using the  classical results for smooth initial data presented in Section  \ref{classical.smooth}, by
\eqref{error.smoot} we get
$$\|\th \tu-\tuh
^h\|_{L^\infty(0,T;\,l^2(h\zz))}\leq C h^{1/2}\exp(T\|\tu\|_{L^\infty(0,T;\, H^1(\rr))}^p).$$
Thus
\begin{align*}
\|\th u-\tuh\|_{L^\infty(0,T;\,l^2(h\zz))}&\leq \|\th u-\th \tu\|_{L^\infty(0,T;\,l^2(h\zz))}+
\|\th\tu-\tuh\|_{L^\infty(0,T;\,l^2(h\zz))}\\
&\leq \|u-\tu\|_{L^\infty(0,T;\,L^2(\rr))}+\|\th\tu-\tuh\|_{L^\infty(0,T;\,l^2(h\zz))}\\
&\leq C(T,p,\|\varphi\|_{L^2(\rr)}) \|\varphi-\tvarphi\|_{L^2(\rr)}+h^{1/2}\exp(T\|\tu\|_{L^\infty(0,T;\, H^1(\rr))}^p).\end{align*}
This yields \eqref{est.120}.
\end{proof}

\section{Technical Lemmas}\label{putos.lemas}
\setcounter{equation}{0}

In this section we prove some technical results that have been used along the paper. The main aim of this section
 is to obtain estimates on the difference $f(\th u)-\th f(u)$ in auxiliary norms $L^q(I,l^r(h\zz))$.

  In the case of smooth enough functions
 $u$,  the pointwise projection operator
 \begin{equation}\label{defe}
    (\eh u)(jh)=u(jh)
 \end{equation}
makes sense. More precisely it is well defined in $H^s(\rr)$, $s>1/2$.
 In these cases the use of the operator $\eh$ has the advantage of commuting with the nonlinearity $f(\eh u)=\eh f(u)$.

The key ingredient is the following Theorem.
\begin{teo}(Marcinkiewicz multiplier theorem \cite{grafakos}, Th.~5.2.2, p.356)\label{graf}
 Let $m:\rr\rightarrow \rr$ be a bounded function which is $C^1$ in every dyadic set $(2^j,2^{j+1})\cup
(-2^{j+1},-2^{j})$ for $j\in \zz$. Assume that the derivative $m'$ of $m$ satisfies
\begin{equation}\label{mark}
    \sup_{j\in \zz}\left[\int_{-2^{j+1}}^{-2^j}|m'(\xi)|d\xi+\int_{2^{j}}^{2^{j+1}}|m'(\xi)|d\xi\right]\leq A<\infty.
\end{equation}
Then there exists a positive constant $C$ such that for all
 $1< q< \infty$ the following holds:
$$\| (\hat f m)^\vee \|_{L^q(\rr)}\leq C\max\{q,(q-1)^{-1}\}^6(A+\|m\|_{L^\infty(\rr)})\|f\|_{L^q(\rr)}.$$
\end{teo}

\begin{remark}Using a change of variables in the Fourier space the above
 dyadic intervals can be replaced by any other one of the form $(c^j,c^{j+1})\cup
(-c^{j+1},-c^{j})$, $j\in \zz$ and $c>1$. In the following applications, the constant $c$ will be chosen to be $c=\pi$.
\end{remark}

For any function $u\in L^2(\rr)$ we define the new function $\tuh$ by truncating
the Fourier transform as follows:
\begin{equation}\label{utilde}
\what\tuh(\xi)= \what u(\xi)\,{\bf{1}}_{(-\pi/h,\,\pi/h)}(\xi).
\end{equation}
For $h=1$, Theorem \ref{graf}, applied with $m(\xi)={\bf 1}_{(-\pi,\pi)}$ which is $C^1$ in every dyadic interval,  shows that for any $1<q<\infty$, the $L^q(\rr)$-norm of
$\tilde u_1$ can be controlled by the one
of $u$:
\begin{equation}\label{est.lp}
\|\tilde u_1\|_{L^{q}(\rr)}\leq C(q)\|u\|_{L^{q}(\rr)}.
\end{equation}
A scaling argument shows us that the above inequality also holds for all $h>0$ with a constant $C(q)$
independent of $h$.


Using Theorem \ref{graf}  we can refine this estimate  as follows:
\begin{lemma}\label{lemma.1}
For any  $s\geq 0$ and $q\in (1,\infty)$ the following hold. \\
a)
There exists a positive constant $c(s,q)$ such that
\begin{equation}\label{est.1}
    \|u-\tuh\|_{L^q(\rr)}\leq c(s,q)h^s \|u\|_{\dot W^{s,q}(\rr)}
\end{equation}
holds for all $u\in \dot W^{s,q}(\rr)$ and $h>0$.\\
b)
Assuming  $s\in [0,1]$, there exists a positive constant $c(s,q)$ such that
\begin{equation}\label{est.2}
    h\|\tuh\|_{\dot W^{1,q}(\rr)}\leq c(s,q)h^s \|u\|_{\dot W^{s,q}(\rr)}
\end{equation}
holds for all $u\in \dot W^{s,q}(\rr)$ and $h>0$.
\end{lemma}

\begin{proof}[Proof of Lemma \ref{lemma.1}]
We divide the proof in two steps  corresponding to \eqref{est.1} and \eqref{est.2}.

{\bf Step I. Proof of \eqref{est.1}.}
Let us consider the following operator
$$M_h u:=u-\tilde u_h=( {\bf 1}_{\{|\xi|\geq \pi/h\}}\hat u)^\vee.$$
A change of variables gives us that
$$(M_h u)(x)=M_1(u(h\cdot ))\Big(\frac xh\Big).$$
Using this property the following identities hold:
$$\|M_hu\|_{L^q(\rr)}=h^{1/q}\|M_1(u(h\cdot))\|_{L^q(\rr)}$$
and
$$\|u(h\cdot)\|_{\dot W^{s,q}(\rr)}=\||\nabla|^s[u(h\cdot)]\|_{L^q(\rr)}=h^s\|(|\nabla|^s u)(h\cdot)\|_{L^q(\rr)}=
h^sh^{1/q}\||\nabla|^su\|_{L^q(\rr)}.$$
Thus, it is sufficient to consider the case $h=1$ and to prove that
\begin{equation}\label{est.4}
    \|M_1v\|_{L^q(\rr)}\leq c(s,q)\||\nabla|^s v\|_{L^q(\rr)}
\end{equation}
holds for all $v\in \dot W^{s,q}(\rr)$.

With the notation  $$m_s(\xi):={|\xi|^{-s}}{\bf 1}_{\{|\xi|\geq \pi\}}(\xi),$$ estimate \eqref{est.4} holds if
$m_s(\xi)$ satisfies the hypothesis of Theorem \ref{graf}. Using that $m_s(\xi)\in L^\infty(\rr)$  and that
$$|m'_s(\xi)|\leq \frac{c(s)}{|\xi|^{s+1}}{\bf 1}_{\{|\xi|\geq \pi\}}(\xi),\, \xi\in \rr,$$
by Theorem \ref{graf} we obtain \eqref{est.4}.

\medskip
{\bf Step II. Proof of \eqref{est.2}.}
A similar argument as in the previous case reduces estimate \eqref{est.2} to the case $h=1$:
$$\|(\what u (\xi){\bf 1}_{(-\pi,\pi)}|\xi|)^\vee\|_{L^q(\rr)}\leq c(s,q)
\|(\what u(\xi) |\xi|^s)^\vee\|_{L^q(\rr)}.
$$
Denoting $v=(\what u(\xi) |\xi|^s)^\vee$, it remains to prove that
\begin{equation}\label{est.6}
    \|(\hat{v}(\xi){\bf 1}_{(-\pi,\pi)}|\xi|^{1-s})^\vee\|_{L^q(\rr)}\leq c(s,q)\|v\|_{L^q(\rr)}.
\end{equation}
In other words, it is sufficient to apply Theorem \ref{graf} to the
 multiplier $m_s(\xi)$ given by
$$m_s(\xi)=|\xi|^{1-s}{\bf 1}_{(-\pi,\pi)}(\xi).$$
Using that $m_s(\xi)\in L^\infty(\rr)$  
satisfies
$$|m_s'(\xi)|\leq c(s)|\xi|^{-s}{\bf 1}_{(-\pi,\pi)}(\xi), \, \xi\in \rr\setminus \{0\},$$
 we fit in the hypothesis of Theorem  \ref{graf} and then \eqref{est.6} holds. This finishes the proof.
\end{proof}

In the following we obtain error estimates for the difference between the two interpolators  $\th $  and $\eh $ when applied to
functions $u$ and $f(u)$, where $\th $  and $\eh $ are defined by
\eqref{def.th} and \eqref{defe} respectively.

\begin{lemma}\label{lemma.3}
Let $s>1/2$ and $q\in (1, \infty)$. Then there exists a positive constant $c(s,q)$ such that
\begin{equation}\label{est.3}
    \|\th u-\eh u\|_{l^q(h\zz)}\leq c(s,q)h^s \|u\|_{\dot W^{s,q}(\rr)}
\end{equation}
holds for all $u\in W^{s,q}(\rr)$ and $h>0$.
\end{lemma}

\begin{remark}
This lemma generalizes Theorem 10.1.3 of  \cite{MR1005330}, p. 205, which addresses the case $q=2$, $s>1/2$. In this case
using  Plancherel's identity in the discrete setting it is easy to obtain
\begin{equation}\label{est.l2}
    \|\th u-\eh u\|_{l^2(h\zz)}\leq c(s)h^s \|u\|_{\dot H^{s}(\rr)}.
\end{equation}
\end{remark}

\begin{remark}
Using the above results,  we will be able to obtain estimates of the difference $\th f(u)-f(\th u)$, $f(u)=|u|^pu$, $p\geq 0$,
given by Lemma \ref{lemma.non}.
\end{remark}

\begin{proof}[Proof of Lemma \ref{lemma.3}]
Estimate \eqref{est.l2} provides the desired estimate $\dot W^{s,2}(\rr)\rightarrow l^2(h\zz)$ in  the case $q=2$.
We will also prove the estimate
$\dot W^{s,q}\rightarrow l^q(h\zz)$ in the case $s>1$. Using these two estimates  the general case will be a consequence of an
interpolation argument.

{\bf Case 1: $s>1$, $q\in (1,\infty)$.}
We claim that
\begin{equation}\label{est.23}
    \|\th u-\eh u\|_{l^p(h\zz)}\leq c(p,s)h^s  \| |\nabla|^s u\|_{L^p(\rr)}.
\end{equation}
By rescaling all the above quantities we can assume $h=1$.

We have the following:
$$(\mathbf{T}_1 u-\mathbf{E}_1 u) (j)=\int _{|\xi|\geq \pi} e^{ij\xi}\hat u(\xi)=\int_{-\pi}^{\pi}
e^{ij\xi} \sum _{l\neq 0} \hat u(\xi+2\pi l).$$
Denoting by $v$ the function whose Fourier transform is given by
\begin{equation}\label{defv}
\hat v(\xi)={\bf 1}_{(-\pi,\pi)}\sum _{l\neq 0} \hat u(\xi+2\pi l),
\end{equation}
we get
$$(\mathbf{T}_1 u-\mathbf{E}_1 u) (j)=\int_{-\pi}^{\pi}e^{ij\xi}\hat v(\xi) d\xi.$$
Classical results on band-limited functions
(see Plancherel-Polya \cite{0018.15204}) give us that
$$\|\mathbf{T}_1 u-\mathbf{E}_1 u\|_{l^p(\zz)}\leq \|v\|_{L^p(\rr)},$$
provided that the right hand side term of the above inequality makes sense. It is then sufficient to prove that the
function $v$ defined by \eqref{defv} satisfies:
 \begin{equation}\label{newestimateonv}
    \|v\|_{L^p(\rr)}\leq c(p,s)\| |\nabla|^s u\|_{L^p(\rr)}.
 \end{equation}

Using the properties of the Fourier transform we get:
$$v(x)=\sum _{l\neq 0}e^{2i\pi lx  } ( {\bf {1}}_{((2l-1)\pi,(2l+1)\pi)}\hat u)^\vee.$$
It is sufficient to prove that
$$\left\|\sum _{l\neq 0}e^{2i\pi lx  } ( {\bf {1}}_{((2l-1)\pi,(2l+1)\pi)}\hat u)^\vee\right\|_{L^p(\rr)}\leq
\||\nabla|^s u |\|_{L^p(\rr)}$$
or equivalently
$$\left\|\sum _{l\neq 0}e^{2i\pi lx  } ( |\xi|^{-s}{\bf {1}}_{((2l-1)\pi,(2l+1)\pi)}\hat u )^\vee\right\|_{L^p(\rr)}\leq
\| u\|_{L^p(\rr)}.$$
Minkowsky's inequality gives us
\begin{align*}
\left\|\sum _{l\neq 0}e^{2i\pi lx  } ( |\xi|^{-s}{\bf {1}}_{((2l-1)\pi,(2l+1)\pi)}\hat u )^\vee\right\|_{L^p(\rr)}\leq
\sum _{l\neq 0}\left\|(  |\xi|^{-s}{\bf {1}}_{((2l-1)\pi,(2l+1)\pi)}\hat u)^\vee\right\|_{L^p(\rr)}.
\end{align*}
We claim that for any $l\neq 0$:
\begin{equation}\label{claim1}
\left\|(  |\xi|^{-s}{\bf {1}}_{((2l-1)\pi,(2l+1)\pi)}\hat u)^\vee\right\|_{L^p}\leq
\frac {c(s)}{|l|^s}\|u\|_{L^p(\rr)}.
\end{equation}
Thus, summing all the above inequalities for $l\neq 0$ we obtain the desired estimate.

A translation in \eqref{claim1} reduces its proof to show that $m_{s,l}$, defined by
$$m_{s,l}(\xi)=|\xi-2l\pi|^{-s}{\bf {1}}_{(-\pi,\pi)}(\xi), \, l\neq 0,$$ verify the hypothesis of
Proposition \ref{graf}.
Observe that
$$|m_{s,l}(\xi)|\leq  \frac {c(s)}{|l|^s},\, \xi\in \rr, l\neq 0$$
and
$$|m'_{s,l}( \xi)|\leq \frac {c(s)}{|l|^{s}|\xi|}{\bf {1}}_{(-\pi,\pi)}(\xi),\, \xi\in \rr\setminus\{0\}, l\neq 0.$$
Applying Proposition \ref{graf} to each multiplier  $m_{s,l}$ we get \eqref{claim1} and the proof of this case is finished.

\medskip
{\bf Case 2: $s>1/2$, $q\in (1,\infty)$.} We set $U_h=\th-\eh$.
Using the estimates of the previous case  we deduce that the operator $U_h$ satisfies:
$$U_h:\dot W^{s_1,q_1}(\rr)\rightarrow l^{q_1}(h\zz),\, s_1>1,\, 1<q_1<\infty,$$
and by \eqref{est.l2}:
$$U_h:\dot W^{s_2,2}(\rr)\rightarrow l^2(h\zz),\, s_2>1/2.$$
Then for any $\theta\in (0,1)$,
$$U_h:[\dot W^{s_1,q_1}(\rr), \dot W^{s_2,2}(\rr)]_{[\theta]}\rightarrow [l^{q_1}(h\zz),l^2(h\zz)]_{[\theta]}$$
with a norm that satisfies:
$$ \|U_h\|_{[\dot W^{s_1,q_1}(\rr), \dot W^{s_2,2}(\rr)]_{[\theta]}- [l^{q_1}(h\zz),\,l^2(h\zz)]_{[\theta]}}
 \leq \|U_h\|_{\dot W^{s_1,q_1}(\rr)-l^{q_1}(h\zz)}^\theta \|U_h\|_{\dot W^{s_2,2}(\rr)-l^2(h\zz)}^{1-\theta}.$$
Classical results on interpolation theory (\cite{0344.46071} ,Th.~6.4.5, p.~153) give us that
$$[\dot W^{s_1,q_1}(\rr), \dot W^{s_2,2}(\rr)]_{[\theta]}=\dot W^{s,q}(\rr)$$
and
$$[l^{q_1}(h\zz),l^2(h\zz)]_{[\theta]}=l^{q}(h\zz)$$
where $s$ and $q$ are given by
\begin{equation}\label{rel.sq}
    \left\{ \begin{array}{l}
    s=f_\theta(s_1,s_2)=s_1\theta+s_2(1-\theta),\\[8pt]
\displaystyle    \frac 1q=g_\theta(q_1)=\frac \theta {q_1}+\frac{1-\theta}{2}.
    \end{array}
    \right.
\end{equation}
Using that the ranks of functions $f_\theta$ and $g_\theta$ satisfy
$$Im(f_\theta)=\Big(\frac{1+\theta}2,\infty\Big),\, Im(g_\theta)=\Big(\frac{1-\theta}2,\frac {1+\theta}2\Big),$$
we obtain that for any $s>1/2$ and $0<q<1$ we can find $s_1>1$, $s_2>1/2$, $q_1>1$ and $\theta \in (0,1)$ such that
\eqref{rel.sq} holds and
\begin{equation}\label{est.fin}
    \|A_h\|_{\dot W^{s,q}(\rr)-l^q(h\zz)}\leq h^{s_1\theta} h^{s_2(1-\theta)}\leq h^s.
\end{equation}The proof is now finished.
\end{proof}

\begin{proof}[Proof of Lemma \ref{lemma.non}]
We first recall that the following inequality holds for all $u,v\in L^{p+2}(\rr)$:
\begin{equation}\label{ineg.1}
\|f(u)-f(v)\|_{L^{(p+2)'}(\rr)}\leq C(p)(\|u\|_{L^{p+2}}^p+\|v\|_{L^{p+2}}^p)\|u-v\|_{L^{p+2}(\rr)}.
\end{equation}
We set $\tilde u_h$ defined by $\what\tuh(\xi)= \what u(\xi)\,{\bf{1}}_{(-\pi/h,\,\pi/h)}(\xi)$.
The difference $\th f(u)-f(\th u)$  in \eqref{inter.1} satisfies:
\begin{equation*}\label{inter.2}
    \|\th f(u)-f(\th u)\|_{l^{(p+2)'}(h\zz)}\leq \|\th f(u)-\th f(\tuh)\|_{l^{(p+2)'}(h\zz)}+
    \|\th f(\tuh)-f(\th u)\|_{l^{(p+2)'}(h\zz)}.
\end{equation*}
Using \eqref{ineg.1}, \eqref{est.lp} and Lemma \ref{lemma.1}, the first term in the right hand side satisfies
\begin{eqnarray*}
  \|\th f(u)-\th f(\tuh)\|_{l^{(p+2)'}(h\zz)} &\leq & c(p)\|f(u)-f(\tuh)\|_{L^{(p+2)'}(\rr)} \\
   &\leq & c(p)(\|u\|_{L^{p+2}(\rr)}^{p}+\|\tuh\|_{L^{p+2}(\rr)}^{p})\|u-\tuh\| _{L^{p+2}(\rr)}\\
   &\leq & c(p)h^s\|u\|_{L^{p+2}(\rr)}^{p}  \|u\|_{\dot W^{s,p+2}(\rr)}\leq c(p) h^s \|u\|_{ W^{s,p+2}(\rr)}^{p+1}.
\end{eqnarray*}
For the second term, using that on the grid $h\zz$, $\th u=\eh \tuh$, by Lemma \eqref{lemma.3} we get:
\begin{eqnarray}
  \label{est.20}  \|\th f(\tuh)-f(\th u)\|_{l^{(p+2)'}(h\zz)}&=&  \|\th f(\tuh)-f(\eh \tuh)\|_{l^{(p+2)'}(h\zz)}\\
  \nonumber &=&   \|\th f(\tuh)-\eh  f(\tuh)\|_{l^{(p+2)'}(h\zz)}\\
 \nonumber &\leq &h \|f(\tuh)\|_{\dot W^{1,(p+2)'}(\rr)}\leq h \|\tuh ^p \partial _x\tuh \|_{L^{(p+2)'}(\rr)}.
\end{eqnarray}

Using that $s\in [0,1]$ we apply Young's inequality and \eqref{est.2} to obtain:
\begin{align}
 \label{est.21}\|\tuh ^p \partial _x\tuh \|_{L^{(p+2)/(p+1)}(\rr)}&= \Big(\int _{\rr}
|\tuh|^{p(p+2)/(p+1)}|\partial_x \tuh|^{(p+2)/(p+1)}\Big)^{(p+1)/(p+2)}\\
 \nonumber&\leq \Big(\||\tuh|^{p(p+2)/(p+1)}\|_{(p+1)/p}\||\partial_x \tuh|^{(p+2)/(p+1)}\|_{p+1}\Big)^{(p+1)/(p+2)}\\
 \nonumber&= \|\tuh\|_{L^{p+2}(\rr)}^{p}\|\partial_x \tuh\|_{L^{p+2}(\rr)}\leq \|u\|_{L^{p+2}(\rr)}^{p} \|\tuh\|_{\dot
W^{1,p+2}}\\
\nonumber&\lesssim \|u\|_{L^{p+2}(\rr)}^{p}h^{s-1} \|u\|_{\dot W^{s,p+2}(\rr)}\leq
h^{s-1} \|u\|_{W^{s,p+2}(\rr)}^{p+1}.
\end{align}
Thus by \eqref{est.20} and \eqref{est.21} we obtain
$$\|\th f(\tuh)-f(\th u)\|_{l^{(p+2)'}(h\zz)}\leq h^s \|u\|_{W^{s,p+2}(\rr)}^{p+1}$$
which finishes the proof.
\end{proof}

 {\bf
Acknowledgements.}

The first author was partially supported by Grant PN-II-ID-PCE-2011-3-0075 of the Romanian National Authority for Scientific Research, CNCS--UEFISCDI, MTM2011-29306-C02-00, MICINN, Spain and ERC Advanced Grant FP7-246775 NUMERIWAVES.  The second author
was partially supported by Grant MTM2011-29306-C02-00, MICINN, Spain, ERC Advanced Grant FP7-246775 NUMERIWAVES, ESF Research Networking Programme OPTPDE and Grant PI2010-04 of the Basque Government.

This work was started when the authors were visiting the Isaac Newton Institute,
Cambridge, within the program "Highly Oscillatory Problems".
The authors also acknowledge this institution and Professor A. Iserles for their hospitality and support.

\end{document}